\documentclass[11pt]{article}
\usepackage{color,latexsym,amsthm,amsmath,amssymb,path,enumitem,url}

\newcommand{\ignore}[1]{}

\newtheorem{theorem}{Theorem}[section]
\newtheorem{lemma}[theorem]{Lemma}
\newtheorem{corollary}[theorem]{Corollary}

\newcommand{\Proof}[1]
        {
        \noindent
        \emph{Proof #1.}~
        }
\newsavebox{\smallProofsym}                     
\savebox{\smallProofsym}                        %
        {
        \begin{picture}(7,7)                    %
        \put(0,0){\framebox(6,6){}}             %
        \put(0,2){\framebox(4,4){}}             %
        \end{picture}                           %
        }                                       %
\newcommand{\smalleop}[1]
        {
        \mbox{} \hfill #1~~\usebox{\smallProofsym}\!\!\!\!\!\!\
        }

\newcommand{\parag}[1]{\vspace{2mm}

\noindent{\bf #1} }

\usepackage{graphicx,psfrag}
\usepackage[rflt]{floatflt}

\usepackage{dsfont}
\newcommand{\NN}{\ensuremath{\mathbb N}}
\newcommand{\ZZ}{\ensuremath{\mathbb Z}}
\newcommand{\RR}{\ensuremath{\mathbb R}}
\newcommand{\CC}{\ensuremath{\mathbb C}}

\newcommand{\FF}{\ensuremath{\mathbb F}}
\newcommand{\DD}{\ensuremath{\mathbb D}}
\newcommand{\SSS}{\ensuremath{\mathbb S}}

\newcommand{\pts}{\mathcal P}
\newcommand{\planes}{\Pi}

\newcommand{\lines}{\mathcal L}

\newcommand{\SL}{{\mathrm{SL}}}

\def\re{\text{Re}}
\global\long\def\r#1{r_{A}^{\div}\left(#1\right)}

\newcommand{\bi}{{\bf i}}

\def\eps{{\varepsilon}}

\begin{document}
\pagenumbering{arabic}

\title{Sum-Product Phenomena for Planar Hypercomplex Numbers\footnote{This research project was done as part of the 2018 CUNY Combinatorics REU, supported by NSF grant DMS-1710305.}}

\author{
Matthew Hase-Liu\thanks{Harvard University, Cambridge, MA, USA. {\sl matthewhaseliu@college.harvard.edu}}
\and
Adam Sheffer\thanks{Department of Mathematics, Baruch College, City University of New York, NY, USA.
{\sl adamsh@gmail.com}. Supported by NSF award DMS-1710305 and PSC-CUNY award 61666-00-49.}}

\date{}
\maketitle
\begin{abstract}
We study the sum-product problem for the planar hypercomplex numbers: the dual numbers and double numbers.
These number systems are similar to the complex numbers, but it turns out that they have a very different combinatorial behavior.
We identify parameters that control the behavior of these problems, and derive sum-product bounds that depend on these parameters.
For the dual numbers we expose a range where the minimum value of $\max\{|A+A|,|AA|\}$ is neither close to $|A|$ nor to $|A|^2$.

To obtain our main sum-product bound, we extend Elekes' sum-product technique that relies on point-line incidences.
Our extension is significantly more involved than the original proof, and in some sense runs the original technique a few times in a bootstrapping manner.
We also study point-line incidences in the dual plane and in the double plane, developing analogs of the Szemer\'edi--Trotter theorem.
As in the case of the sum-product problem, it turns out that the dual and double variants behave differently than the complex and real ones.

\end{abstract}

\section{Introduction}

It is not uncommon for a combinatorial problem to be defined over $\RR$, to then be generalized to $\CC$, and then further generalized to the quaternions.
For example, this is the case for the sum-product problem \cite{BL18,KR13,SW17}, for geometric incidence problems  \cite{SSZ18,SS08,ST12}, and for the Sylvester--Gallai problem \cite{EPS06}.

Both the complex numbers and the quaternions are types of \emph{hypercomplex numbers}.
A system of hypercomplex numbers is a unital algebra with every element having the form
\[ a_0 + a_1 \bi_1 + a_2 \bi_2 + \cdots + a_n \bi_n. \]
Here $n\in \NN$, $a_0,a_1,\ldots,a_n\in \RR$, and $\bi_1,\ldots,\bi_n$ are called \emph{imaginary units}.
To be an algebra over the reals, the system also needs to include a multiplication table for the imaginary units.
The \emph{dimension} of the system is $n+1$, agreeing with the standard definition of the dimension of a vector space.
For example, the complex numbers are a two-dimensional system with the multiplication rule $\bi^2=-1$.
The quaternions form a system of dimension four and involves a $3\times3$ multiplication table for the three imaginary units.
For a nice basic introduction to hypercomplex numbers, see for example \cite{KS89}.

We refer to two-dimensional systems of hypercomplex numbers as \emph{planar}.
Up to isomorphisms, there are exactly three such planar systems: The complex numbers, the \emph{dual numbers}, and the \emph{double numbers} (for a proof of this claim, see for example \cite[Section 2]{KS89}).
The dual numbers are of the form $a+b\eps$, where $a,b\in \RR$, and with the multiplication rule $\eps^2=0$.
The double numbers are of the form $a+bj$ with the multiplication rule $j^2=1$.
Double number are often also called split-complex numbers, and more generally have at least 18 different names in the literature (Clifford referred to them as \emph{algebraic motors}, and some other names are \emph{spacetime numbers} and \emph{anormal-complex numbers}).

The dual and double numbers seem to appear in many different fields.
For example, they are used in String Theory \cite{GGP96}, Kinematics \cite{Fischer00}, and Signal Processing \cite{MG09}.
Dual numbers play a role in the theory of schemes (for example, see \cite{Hars83}). 
They are used in geometry, and we were originally introduced to them through Kisil's lecture notes on the Erlangen program \cite{Kisil12}.
Double numbers were even used to design algorithms for dating sites \cite{KGG12}.
However, to the best of our knowledge, dual and double numbers were not seriously studied from a combinatorial perspective.
The goal of the current work is initiate such a combinatorial study.

We study a variant of the sum-product problem for dual and double numbers.
To obtain sum-product results, we also study other combinatorial properties of these number systems.
In particular, we derive variants of the Szemer\'edi--Trotter theorem for dual and double numbers.

Beyond initiating a combinatorial study of dual and double numbers, we believe that our results are also of intrinsic interest to the study of sum-product phenomena.
The sum-product problem seems to have a similar behavior over the reals, over the complex, and over the quaternions.
Results over the reals are usually extended to the complex and to the quaternions.
Surprisingly, the sum-product problem has a significantly different behavior over the dual numbers.
In addition, our main technique is based on a new idea: Using Elekes' sum-product technique several times, each time relying on the previous result in a bootstrapping manner.

\parag{The sum product problem.}
Given a finite set $A\subset \RR$, the \emph{sum set} and \emph{product set} of $A$ are respectively defined as
\[ A+A = \{a+a' : a,a'\in A\} \quad \text{ and } \quad AA = \{a\cdot a' : a,a'\in A\}. \]

Erd\H os and Szemer\'edi \cite{ES83} conjectured that for every $\rho>0$, any sufficiently large $A\subset \NN$ satisfies
$\max⁡\{|A+A|,|AA|\} =\Omega\left(|A|^{2-\rho}\right)$. (Since $\eps$ is already taken and $\delta$ is also used in our analysis, throughout the paper we will use $\rho$ as a small positive real number.)
The problem was later generalized to sets of real numbers, sets of complex numbers, quaternions, finite fields, and more.
The problem remains wide-open for all of these variants.
For the case of $A\subset \RR$, in 2009 Solymosi \cite{Soly09} proved the bound $\max⁡\{|A+A|,|AA|\} =\Omega^*\left(|A|^{4/3}\right)$. In the $\Omega^*(\cdot)$-notation, we neglect subpolynomial factors such as $\log |A|$ and $2^{\sqrt{\log|A|}}$. The same holds for the $O^*(\cdot)$-notation and for the $\Theta^*(\cdot)$-notation.
After a series of improvements, the current best bound over $\RR$, derived in \cite{Shakan18}, is $\Omega^*\left(|A|^{4/3+5/5277}\right)$.
Similar bounds exist for the complex numbers and for the quaternions (for example, see \cite{BL18,SW17}).
The best know upper bound for all of these variants is $O^*\left(|A|^2\right)$.

\parag{Dual numbers.}
Let $\DD$ be the set of dual numbers: The extension of $\RR$ with the extra element $\eps$ and the rule $\eps^2=0$.
We write a number $a\in \DD$ as $a_1+\eps a_2$.
Imitating the complex numbers, we refer to $a_1$ as the \emph{real} part of $a$, and to $a_2$ as the \emph{imaginary} part of $a$.
Unlike $\RR$ and $\CC$, the dual numbers do not form a field, since some dual numbers are not invertible.
In particular, a dual number has an inverse if and only if it has a non-zero real part.

Unlike the cases of $\RR$, $\CC$, and the quaternions, the sum--product conjecture is false over the dual numbers.
For example, consider the set
\[ A = \{ 1 + \eps m \in \DD \ :\ m\in \ZZ,\ 1\le m \le n\},\]
and note that
\[ A+A = \{ 2+\eps m\ :\ m\in \ZZ,\ 2\le m \le 2n\} \quad \text{ and } \quad AA = \{ 1+\eps m\ :\ m\in \ZZ,\ 2\le m \le 2n\}. \]
That is, the sizes of both the sum set and product set are linear in $|A|$.
Note that all of the elements of $A$ are invertible, and so are the elements of $A+A$ and of $AA$.
When allowing non-invertible elements, one can have a product set of \emph{size one} and a linear-sized sum set.
However, we are not interested in constructions that are based on non-invertible elements.

It turns out that the maximum number of elements of $A$ that have the same real part plays an important role.
We say that a set $A\subset \DD$ has \emph{multiplicity} $k$ if every real number is the real part of at most $k$ elements of $A$.
We usually denote the size of $A$ as $n$ and the multiplicity of $A$ as
$n^{\alpha}$, for some $0\le \alpha \le 1$.

To adapt the above construction to the case of multiplicity $n^\alpha$, we consider the set
\[ A = \{ a_1 + \eps a_2  \ :\ a_1,a_2\in \ZZ,\ 1\le a_1 \le n^{1-\alpha}, \ 1\le a_2 \le n^{\alpha}\},\]
and note that
\begin{align*}
A+A &= \{ m_1+\eps m_2\ :\ m_1,m_2\in \ZZ,\ 2\le m_1 \le 2n^{1-\alpha},\ 2\le m_2 \le 2n^{\alpha} \},\\
AA &\subset \{ m_1+\eps m_2\ :\ m_1,m_2\in \ZZ,\ 1\le m_1 \le n^{2-2\alpha}, \ 2\le m_2 \le 2n \}.
\end{align*}

Note that in this construction $A$ indeed has multiplicity $n^\alpha$.
The size of the sum set is $\Theta(|A|)$ and the size of the product set is $O(|A|^{3-2\alpha})$.
Thus, the sum-product conjecture is false for any $\alpha>1/2$.
On the other hand, we show that $\max \{|A+A|,|AA|\}$ is super-linear in $|A|$ when $\alpha$ is not too large.
Set $\kappa =  (39 - \sqrt{721})/20 \approx 0.607$.

\begin{theorem} \label{th:DualSP}
Let $A$ be a set of $n$ dual numbers with multiplicity $n^\alpha$, for some $0\le \alpha<\kappa$.
Then for every $\rho>0$,
\[ \max \left\{|A+A|,|AA|\right\} = \begin{cases}
\Omega^*\left(n^{(4-2\alpha)/3}\right), & \quad 0\le \alpha < 1/8, \\[2mm]
\Omega\left(n^{5/4-\rho}\right), & \quad 1/8\le \alpha<1/3, \\[2mm]
\Omega^*\left(n^{3/2-5\alpha/8}\right), & \quad 1/3\le\alpha<1/2,\\[2mm]
\Omega\left(n^{9/4-39\alpha/16+5\alpha^2/8-\rho}\right), & \quad 1/2\le\alpha<\kappa.
\end{cases}\]
\end{theorem}

Combining Theorem \ref{th:DualSP} with the above construction leads to a surprising observation: When $1/2<\alpha<\kappa$ the bound for the sum-product problem is neither $\Theta^*(n^2)$ nor $\Theta(n)$.
It is hard to guess what the actual value should be.
One possibility is that the sum-product conjecture holds when $\alpha\le 1/2$, and is replaced with $\Theta^*(n^{3-2\alpha})$ when $\alpha>1/2$.

Different bounds in the statement of Theorem \ref{th:DualSP} are obtained using different approaches.
The bound for the range $0\le \alpha < 1/8$ is obtained from a relatively simple adaptation of Solymosi's technique from \cite{Soly09}.
Surprisingly, we were able to obtain a stronger bound when $\alpha\ge 1/8$ by relying on an earlier approach of Elekes \cite{Elekes97}.
While we use Elekes' approach, our technique is significantly more involved, and is the main result of this work.
In addition to having several new steps, we use Elekes' approach several times, each time relying on the result of the previous case.

Our extension of Elekes' technique leads to the bounds of Theorem \ref{th:DualSP} for range $1/8\le \alpha<1/3$ and also for the range $1/3\le\alpha<1/2$.
This technique breaks down when $\alpha \ge 1/2$.
The bound for the range of $1/2\le\alpha<\kappa$ is obtained by a naive approach --- removing elements from $A$ to decrease the multiplicity to $1/2-\rho$, and then applying the bound of Theorem \ref{th:DualSP} for the range $1/3\le\alpha<1/2$.

\parag{Double numbers.}
Let $\SSS$ be the set of double numbers: The extension of $\RR$ with the extra element $j$ and the rule $j^2=1$.
We write a number $a\in \SSS$ as $a_1+j a_2$.
As before, we refer to $a_1$ as the \emph{real} part of $a$, and to $a_2$ as the \emph{imaginary} part of $a$.
The double numbers do not form a field, since some double numbers are not invertible.

Unlike the case of a dual numbers, we could not find a counterexample to the sum-product conjecture in the case of double numbers.
To see some other surprising behavior of the double numbers, consider the sets
\[ A = \{ m+j m :\ 1\le m \le n \} \quad \text{ and } \quad B = \{ m'-j m' :\ 1\le m' \le n \}.\]
Note that $AB = \{0\}$.
That is, the product set of two large sets could be of size one.
This example will not be very relevant for us, since it heavily relies on non-invertible elements.

We say that a set $A\subset \SSS$ has multiplicity $k$ if for every $r\in \RR$ at most $k$ elements $a_1 + ja_2\in A$ satisfy $a_1+a_2=r$ and at most $k$ satisfy $a_1-a_2=r$.
Recall that $\kappa =  (39 - \sqrt{721})/20 \approx 0.607$.
We derive the following sum-product bound for double numbers.

\begin{theorem} \label{th:DoubleSP}
Let $A$ be a set of $n$ double numbers with multiplicity $n^\alpha$, for some $0\le \alpha<\kappa$.
Then for every $\rho>0$,
\[ \max \left\{|A+A|,|AA|\right\} = \begin{cases}
\Omega^*\left(n^{(4-2\alpha)/3}\right), & \quad 0\le \alpha < 1/8, \\[2mm]
\Omega\left(n^{5/4-\rho}\right), & \quad 1/8\le \alpha<1/3, \\[2mm]
\Omega^*\left(n^{3/2-5\alpha/8}\right), & \quad 1/3\le\alpha<1/2,\\[2mm]
\Omega\left(n^{9/4-39\alpha/16+5\alpha^2/8-\rho}\right), & \quad 1/2\le\alpha<\kappa.
\end{cases}\]
\end{theorem}

While Theorem \ref{th:DoubleSP} contains the same bounds as Theorem \ref{th:DualSP}, the proof of Theorem \ref{th:DoubleSP} is more involved.
In some sense it is easier to study dual numbers than double numbers.
That is why we first prove Theorem \ref{th:DualSP} in Section \ref{sec:Dual}, and then prove Theorem \ref{th:DoubleSP} in Section \ref{sec:Double}.

\parag{Point-line incidences.}
Our sum-product technique requires studying point-line incidences in the plane.
Thus, we first study analogs of the Szemer\'edi--Trotter theorem in $\DD^2$ and in $\SSS^2$.

Given a set $\pts$ of points and a set $\lines$ of lines in $\RR^2$, an \emph{incidence} is a pair $(p,\ell)\in \pts \times \lines$ such that the point $p$ is contained in the line $\ell$.
The number of incidences in $\pts\times\lines$ is denoted as $I(\pts,\lines)$.

\begin{theorem}[The Szemer\'edi-Trotter theorem \cite{ST83}] \label{th:ST83}
Let $\pts$ be a set of $m$ points and let $\lines$ be a set of $n$ lines, both in $\RR^2$.
Then
\[ I(\pts,\lines)=O\left(m^{2/3}n^{2/3}+m+n\right). \]
\end{theorem}

As shown in \cite{ST12,Toth14,Zahl16}, Theorem \ref{th:ST83} still holds when replacing $\RR^2$ with $\CC^2$.
A common approach for incidences in the complex plane is to think of $\CC^2$ as $\RR^4$, obtaining an incidence problem between points and two-dimensional planes.
In particular, the following is a special case of a result of Solymosi and Tao \cite{ST12}.
Consider a point set $\pts$ and a set $\planes$ of two-dimensional planes, both in $\RR^4$.
We say that an incidence $(p,h) \in \pts\times \planes$ is \emph{generic} if there is no additional incidence $(p,h') \in \pts\times \planes$ such that $h\cap h'$ is a line.
That is, two planes that form generic incidences with the same point do not have any other intersection points.

\begin{theorem} \label{th:IncR4}
Let $\pts$ be a set of $m$ points and let $\planes$ be a set of $n$ arbitrary two-dimensional planes, both in $\RR^4$.
Then for every $\rho>0$, the number of generic incidences in $\pts\times \planes$ is
\[ O\left(m^{2/3+\rho}n^{2/3}+m+n\right). \]
\end{theorem}

As we will see below, Theorem \ref{th:ST83} cannot be extended to $\DD^2$ and to $\SSS^2$.
In either case, one can construct a configuration of $m$ points and $n$ lines with $mn$ incidences.
As in the sum-product problem, this maximum number of point-line incidences is controlled by the notion of multiplicity.

We begin with the case of the dual plane.
There are several different ways to define the multiplicity in a point-line incidence problem in $\DD^2$, and we only present one example here.
One may use Lemma \ref{le:LineFamilyDual} to obtain other similar results.
Let $\lines$ be a set of $n$ lines in $\DD^2$, each defined by an equation of the form $y=ax+b$ with $a,b\in \DD$.
We say that $\lines$ has multiplicity $k$ if for every $r\in \RR$, at most $k$ lines of $\lines$ satisfy $a_1=r$.
We also let $\lines$ contain any number of lines of the form $x=b$, without this affecting the multiplicity of the set.
(Our incidence bound also holds when allowing $\lines$ to contain $k$ lines of the form $ax=b$ with $a_1=0$.)

\begin{theorem} \label{th:DualST}
Let $\pts$ be a set of $m$ points and let $\lines$ be a set of $n$ lines, both in $\DD^2$.
Let $\lines$ have multiplicity $n^\alpha$ for some $0\le \alpha \le 1$.
Then for every $\rho>0$,
\[ I(\pts,\lines)=O\left(m^{2/3+\rho}n^{(2+\alpha)/3}+mn^{\alpha}+n\right). \]
\end{theorem}

As we see in Section \ref{ssec:DualLines}, the term $mn^{\alpha}$ is tight and cannot be removed from the bound of Theorem \ref{th:DualST}.
We can also obtain a construction with $m^{2/3}n^{2/3}$ incidences.
The $\rho$ in the bound is almost certainly redundant.
It is much less clear what should be the correct dependency in $\alpha$ in $m^{2/3+\rho}n^{2/3+\alpha/3}$.

We obtain similar incidence results in the double plane $\SSS^2$.
As in the dual case, there are several different ways to define the multiplicity of a point-line incidence problem in $\SSS^2$, and we only present one example.
Let $\lines$ be a set of $n$ lines in $\SSS^2$, each defined by an equation of the form $y=ax+b$ with $a,b\in \SSS$.
We say that $\lines$ has multiplicity $k$ if for every $r\in \RR$, at most $k$ lines of $\lines$ satisfy $a_1+a_2=r$ and at most $k$ such lines satisfy $a_1-a_2=r$.
We also let $\lines$ contain any number of lines of the form $x=b$, without this affecting the multiplicity of the set.
(Our incidence bound still holds also when allowing $\lines$ to contain $k$ lines of the form $ax=b$ with non-invertible $a$.)

\begin{theorem} \label{th:DoubleST}
Let $\pts$ be a set of $m$ points and let $\lines$ be a set of $n$ lines, both in $\SSS^2$.
Let $\lines$ have multiplicity $n^\alpha$ for some $0\le \alpha \le 1$.
Then for every $\rho>0$,
\[ I(\pts,\lines)=O\left(m^{2/3+\rho}n^{(2+\alpha)/3}+mn^{\alpha}+n\right). \]
\end{theorem}

As in the dual plane, the term $mn^{\alpha}$ is tight and cannot be removed from the bound of Theorem \ref{th:DualST}.
We can also obtain a construction with $m^{2/3}n^{2/3}$ incidences.

The Szemer\'edi-Trotter theorem (Theorem \ref{th:ST83}) is considered to have a dual formulation, in the sense that there is a simple combinatorial argument for moving between one formulation and the other.
Given a set of lines $\lines$, we say that a point of $p$ is $r$\emph{-rich} if $p$ is incident to at least $r$ lines of $\lines$.

\begin{theorem}[Dual Szemer\'edi-Trotter] \label{th:DualFormST}
Let $\lines$ be a set of $n$ lines in $\RR^2$, and let $r$ be a positive integer.
Then the number of $r$-rich points of $\lines$ is $O\left(n^2/r^3 + n/r \right)$.
\end{theorem}

\parag{More about the multiplicities. }
Both for the sum-product problem and for the incidence problem, we established that the dual and double variants behave quite differently than the real, complex, and quaternion cases.
An obvious possible explanation for this difference is that $\DD$ and $\SSS$ are not fields.
But these are not fields only because some degenerate numbers have no inverse.
And all of our results hold also when all of the numbers in the problem have inverses.
Moreover, our definitions of multiplicity are not directly about non-invertible elements.

Instead of non-invertible elements in $A$, both definitions of multiplicity ask $A-A$ not to contain many non-invertible elements.
Indeed, in the dual case, $a-a'\in \DD$ is non-invertible when $a_1=a'_1$.
In the double case, $a-a'\in \DD$ is non-invertible when $a_1+a_2=a'_1+a'_2$ or $a_1-a_2=a'_1-a'_2$.
This curious connection between the multiplicity definitions in the dual and double cases might hide a deeper general principle.
In addition, this seems related to a result of Tao \cite[Theorem 5.4]{Tao09}, which holds in a much more general scenario. Vaguely and inaccurately, this result states that a set satisfying $\max\{|A+A|,|AA|\}=\Theta(|A|)$ implies the existence of a linear subspace $V$ of zero-divisors, such that $A$ has a large intersection with a translate of $V$ (see also \cite{KT01}).
This is indeed the situation in our case.
For example, in the dual case $V$ is the line $a_1=0$.
Continuing to expose this hidden principle could potentially be an exciting research front.

\parag{Additional connections to previous sum-product works.}
Similarly to the complex numbers and to the quaternions, one can represent dual and double numbers as matrices.
The standard matrix representations for these numbers are
\[ a_1+\eps a_2 = \left[ {\begin{array}{cc}
   a_1 & a_2 \\
   0 & a_1 \\
  \end{array} } \right] \quad \text{ and } \quad a_1+j a_2 = \left[ {\begin{array}{cc}
   a_1 & a_2 \\
   a_2 & a_1 \\
  \end{array} } \right].
  \]
With these representations, matrix addition and multiplication correspond to the addition and multiplication of dual and double numbers.

With the above matrix representation, our construction of $A\subset \DD$ with $|A+A|=\Theta(|A|)$ and $|AA|=\Theta(|A|)$ corresponds to a construction of Chang \cite{Chang07} for matrices in $\SL(2,\RR)$.
Other papers, such as \cite{SV09,SW17}, study sum-product phenomena for matrices of specific types.
To the best of our knowledge, none of the previous works is relevant to the cases of dual numbers and double numbers.
For example, Theorem 4 of Solymosi and Wong \cite{SW17} depends on the 1-norm of the matrices.
This notion is completely unrelated to our notions of multiplicity.

Recall that when $1/2<\alpha<\kappa$, our sum-product bound for the dual numbers is neither $\Theta^*(n^2)$ nor $\Theta^*(n)$.
A somewhat similar situation was observed before for the sum-product problem in finite fields.
For simplicity, we only consider finite fields $\FF_p$ where $p$ is a prime.
Garaev \cite{Garaev08} constructed a set $A\subset \FF_p$ such that $|A|=\Theta(p^{1/2})$ and $\max\{|A+A|,|AA|\} =O(|A|^{3/2})$.
On the other hand, as shown in \cite{CKM18}, every set $A\subset \FF_p$ with $|A|=O(|A|^{64/117})$ satisfies $\max\{|A+A|,|AA|\} =\Omega(|A|^{39/32})$.

Another elegant argument of Solymosi \cite{Soly05} shows that every finite $A \subset \CC$ satisfies $\max\{|A+A|,|AA|\} =\Omega(|A|^{5/4})$.
The last paragraph of that paper states that ``A similar argument works for quaternions and for other hypercomplex
numbers.''
We now briefly discuss how the current work compares with the results of \cite{Soly05}.
A reader who is not familiar with \cite{Soly05} can safely skip this discussion.

The proof in \cite{Soly05} relies on the standard property that $|a\cdot b| = |a|\cdot |b|$ holds for $a,b\in \CC$ (in particular, this property is used in Lemma 2.1 of \cite{Soly05}).
When working with dual or double numbers, this absolute value property fails when using the standard definition $|a|=\sqrt{a_1^2+a_2^2}$.
In the case of dual numbers, an alternative definition is $|a|=a_1$, which does maintain the property $|a\cdot b| = |a|\cdot |b|$.
When using this definition, a different part of Lemma 2.1 of \cite{Soly05} fails: The claim that no number is covered by more than seven disks.
A similar situation occurs for the double numbers.

Note that the proof of \cite{Soly05} should not hold for dual numbers, since then it would contradict the above construction.
We did manage to get a variant of the argument in \cite{Soly05} to hold for dual and double numbers, while depending on the notion of multiplicity (thus also eliminating the contradiction with the dual construction).
Let $A$ be a set of dual or double numbers with multiplicity $n^\alpha$.
In the proof of Lemma 2.1, instead of being covered by at most 7 disks, no number is covered by more than $2n^\alpha$ ``disks''.
Then, in the definition of good sets one changes the constant 28 with $8n^\alpha$.
Now the proof then holds again, implying the bound $\max\{|A+A|,|AA|\} =\Omega(|A|^{5/4-\alpha/2})$.
It is not difficult to verify that the bounds of Theorem \ref{th:DualSP} and \ref{th:DoubleSP} are stronger for every relevant value of $\alpha$.

\parag{Acknowledgements.}
We would like to thank Misha Rudnev for suggesting this problem, and to Ben Lund, Cosmin Pohoata, and Frank de Zeeuw for helpful discussions.
We would also like to thank J\'ozsef Solymosi --- while he was not even aware of this project, quite a few of his works affected every part of it.

\section{Dual numbers} \label{sec:Dual}

In this section we study dual numbers, and in particular prove Theorem \ref{th:DualSP}.
In Section \ref{ssec:DualLines} we study properties of lines in the dual plane.
We derive a point-line incidence bound in $\DD^2$, and study additional properties of such incidences.
In Section \ref{ssec:ElekDual} we adapt Elekes' sum-product technique to the dual numbers.
As mentioned above, we add several additional steps to Elekes' original argument.
In Section \ref{ssec:SolyDual}, we adapt Solymosi's sum-product argument to the dual numbers.

\subsection{Lines in the dual plane} \label{ssec:DualLines}

Recall that we denote by $\DD$ the set of dual numbers: The extension of $\RR$ with the extra element $\eps$ and the rule $\eps^2=0$.
We write a number $a\in \DD$ as $a_1+\eps a_2$.
Multiplication of dual numbers is commutative, and 1 is the unit element.
A dual number $a\in \DD$ has an inverse element if and only if $a_1\neq 0$.
The inverse element is then $a^{-1} = \frac{a_1-\eps a_2}{a_1^2}$.
Indeed, we have
\[ a\cdot a^{-1} = \frac{(a_1+\eps a_2)(a_1-\eps a_2)}{a_1^2} = \frac{a_1^2}{a_1^2} = 1.\]

We define a line in $\DD^2$ as the set of points on which a linear equation vanishes.
Let $\ell$ be the line defined by $ax+by=c$, where $a,b,c\in \DD$.
This corresponds to
\begin{align*}
(a_1+\eps a_2) (x_1+\eps x_2) + (b_1+\eps b_2)(y_1+\eps y_2) &= (c_1+\eps c_2),
\end{align*}
or equivalently
\begin{align*}
&a_1x_1 \hspace{12.5mm} +  b_1y_1 &= c_1, \\
&a_2x_1 + a_1 x_2 + b_2 y_1 + b_1 y_2 &= c_2.
\end{align*}

When $a_1=b_1=0$, the first equation becomes trivial while the second still exists.
In any other case, the two equations are linearly independent.
We can think of $\DD^2$ as $\RR^4$, and then $\ell$ is either a 2-flat or a hyperplane, depending on whether $a_1=b_1=0$.
We refer to the lines of the latter type as \emph{degenerate lines}.
Note that a line defined by $ax+by=c$ is degenerate if and only if both $a$ and $b$ are non-invertible.

In the real and complex planes, any two lines intersect in at most one point.
In $\DD^2$, two lines can have an infinite intersection, even when excluding non-invertible coefficients in the line equations.
For example, consider the set of non-degenerate lines
\[ \lines = \{ (1+m\eps)x+(1-(m-1)\eps)y = 2+ \eps :\ m\in \RR \}. \]

It is not difficult to verify that every line of $\lines$ contains every point of the form $(1+a\eps,1-a\eps)\in \DD^2$ with $a\in \RR$.
By taking $n$ lines from $\lines$ and $m$ points of the form $(1+a\eps,1-a\eps)\in \DD^2$ we get $mn$ incidences.
That is, the point--line incidence problem in $\DD^2$ is trivial.
This remains true when excluding degenerate lines, and also when using only invertible numbers in the definitions of the points and the lines.
We now study when collections of lines have an infinite intersection.

For $a\in \DD$, denote by $\re(a)$ the real part of $a$.
That is, $\re(a_{1}+a_{2}\eps)=a_{1}$.

\begin{lemma} \label{le:LineFamilyDual}$\quad$ \\
(a) Let $\ell$ and $\ell'$ be distinct lines in $\DD^2$, respectively defined by $y=ax+b$ and $y=a'x+b'$.
Then $\ell\cap\ell'$ contains more than one point if and only if $a_1=a'_1$, $b_1=b'_1$, and $a_2\neq a_2'$.
When these conditions are satisfied, $\ell_1\cap\ell_2$ is a line in $\RR^4$ and there exist $r_1,r_2\in \RR$ such that every point $(x,y)\in \ell_1\cap\ell_2$ satisfies $\re(x)=r_1$ and $\re(y)=r_2$. \\
(b) Let $\lines$ be a set of lines of the form $y=ax+b$ that have an infinite common intersection.
Then all of these lines have the same values for $a_1$ and $b_1$, and there exist $r_1,r_2\in \RR$ such that every point $(x,y)$ in the infinite intersection satisfies $\re(x)=r_1$ and $\re(y)=r_2$.
Moreover, either all the $b_2$ values are identical or there exists $m\in \RR$ such that every line satisfies $b_2=x_1(m-a_2)$.
\end{lemma}
\begin{proof} (a) To study the intersection points of the two lines, we combine $y=ax+b$ and $y=a'x+b'$, obtaining $ax+b=a'x+b'$, or equivalently $x(a-a')=b'-b$.
Splitting this equation into real and imaginary parts gives
\begin{align}
x_1(a_1-a'_1) &= b'_1-b_1, \label{eq:realLineInt} \\
x_1(a_2-a'_2) + x_2(a_1-a'_1) &= b'_2-b_2.  \nonumber
\end{align}

First assume that $a_1\neq a'_1$.
In this case we can rewrite the above system as
\begin{align*}
x_1 &= (b'_1-b_1)/(a_1-a'_1),  \\
x_2 &= (b'_2-b_2-x_1(a_2-a'_2))/(a_1-a'_1).
\end{align*}
Since this system has a unique solution, when $a_1\neq a'_1$ the two lines intersect in a single point.

We next assume that $a_1=a'_1$.
In this case, equation \eqref{eq:realLineInt} implies $b'_1=b_1$.
Then the line equations $y=ax+b$ and $y=a'x+b'$ become
\begin{align*}
y_1 &= a_1x_1 +b_1, \\
y_2 &= a_1x_2+a_2x_1 +b_2, \\
y_2 &= a_1x_2+a'_2x_1 +b'_2.
\end{align*}

Combining the second and third equations of this system gives $a_2x_1 +b_2 = a'_2x_1 +b'_2$.
If $a_2=a'_2$ then the second and third equations of the system imply that either $\ell \cap \ell' = \emptyset$ or $\ell=\ell'$ (depending on whether or not $b_2=b'_2$).
We may thus assume that $a_2\neq a'_2$, to obtain
\[ x_1 = \frac{b'_2-b_2}{a_2-a'_2}, \quad y_1 = a_1\cdot \frac{b'_2-b_2}{a_2-a'_2} +b_1, \quad \text{ and } \quad y_2 = a_1x_2+a_2\cdot \frac{b'_2-b_2}{a_2-a'_2} +b_2. \]

Thus, the intersection $\ell_1\cap\ell_2$ is infinite (it is a line in $\RR^4$).
Moreover, all of the points of $\ell_1\cap\ell_2$ have the same real parts $x_1,y_1$.

(b) By part (a), all the lines in $\lines$ have the same values for $a_1$ and $b_1$, and there exist $r_1,r_2\in \RR$ such that every point $(x,y)$ in the infinite intersection satisfies $\re(x)=r_1$ and $\re(y)=r_2$.
That is, every line of $\lines$ is defined by $y= (a_1+a_2\eps)x+(b_1+b_2\eps)$, where $a_1,b_1\in \RR$ are fixed and $a_2,b_2\in \RR$ change between different lines.
By the proof of part (a), two lines defined by $y= (a_1+\eps a_2)x+(b_1+\eps b_2)$ and $y= (a_1+\eps a'_2)x+(b_1+\eps b'_2)$ satisfy
\[ r_1 = \frac{b'_2-b_2}{a_2-a'_2}. \]

To have every pair of lines of $\lines$ satisfy $b'_2-b_2 = r_1(a_2-a'_2)$, either $r_1 =0$ and then all of the $b_2$ values are identical, or there exists $m\in \RR$ such that every line satisfies $b_2=r_1(m-a_2)$.
\end{proof}

We are now ready to prove Theorem \ref{th:DualST}.
We first recall the statement of this theorem.
\vspace{2mm}

\noindent {\bf Theorem \ref{th:DualST}.}
\emph{Let $\pts$ be a set of $m$ points and let $\lines$ be a set of $n$ lines, both in $\DD^2$.
Let $\lines$ have multiplicity $n^\alpha$ for some $0\le \alpha \le 1$.
Then for every $\rho>0$, }
\[ I(\pts,\lines)=O(m^{2/3+\rho}n^{2/3+\alpha/3}+mn^{\alpha}+n). \]
\begin{proof}
By the definition of multiplicity, the set $\lines$ may contain at most $n^\alpha$ degenerate lines.
Together these lines participate in at most $mn^{\alpha}$ incidences.
Every point of $\pts$ is incident to at most one line of the form $x=b$, so such lines contribute at most $m$ incidences.
It remains to consider incidences with non-degenerate lines of $\lines$ of the form $y=ax+b$.

We discard from $\lines$ the degenerate lines and lines of the form $x=b$.
We can then partition $\lines$ into $n^\alpha$ disjoint subsets $\lines_1,\ldots,\lines_{n^\alpha}$, such that the multiplicity of each $L_j$ is one.
For every $1\le j \le n^\alpha$, set $n_j = |\lines_j|$.
Note that $\sum_{j=1}^{n^\alpha} n_j =n$.
Since each subset has no multiplicity, by Lemma \ref{le:LineFamilyDual} every two lines from the same $\lines_j$ intersect in at most one point.
That is, when thinking of $\DD^2$ as $\RR^4$, the set $\lines_j$ becomes a set of two-dimensional planes, each two intersecting in at most one point. We may thus apply Theorem \ref{th:IncR4} with $\pts$ and $\lines_j$.
Note that in this case every incidence is generic by definition, so Theorem \ref{th:IncR4} gives a bound for the total number of incidences.
By doing that for every $1\le j \le n^\alpha$ and then applying H\"older's inequality, we obtain
\begin{align*}
I(\pts,\lines) &= \sum_{j=1}^{n^\alpha} I(\pts,\lines_j) = \sum_{j=1}^{n^\alpha} O\left(m^{2/3+\rho}n_j^{2/3}+m+n_j\right) \\[2mm]
&= O\left(m^{2/3+\rho}\sum_{j=1}^{n^\alpha} n_j^{2/3}+mn^\alpha+n\right) = O\left(m^{2/3+\rho}n^{2/3+\alpha/3}+mn^\alpha+n\right).
\end{align*}
\end{proof}

For our sum-product results in $\DD$, we need additional properties of point-line incidences in $\DD^2$.
We define the \emph{real part} of $\DD^2$ as the copy of $\RR^2$, and say that a point $(a_1+a_2\eps,b_1+b_2\eps)$ corresponds to the point $(a_1,b_1)$ in the real part of $\DD^2$.
When thinking of $\DD^2$ as $\RR^4$, the real part of $\DD^2$ is the projection of $\RR^4$ to the two real coordinates.
Thus, each point in the real part of $\DD^2$ has an \emph{imaginary plane} associated with it, which is also copy of $\RR^2$.
For example, the point $(1+2\eps,3+4\eps)\in \DD^2$ is the point $(2,4)$ in the imaginary plane associated with the point $(1,3)$ in the real part of $\DD^2$.

We refer to a set of lines in $\DD^2$ with an infinite common intersection as a \emph{line family}.
Let $\lines$ be such a line family.
By Lemma \ref{le:LineFamilyDual}(b), every line of $\lines$ corresponds to the same line in the real part of $\DD^2$.
We refer to this line as the \emph{real line} of $\lines$.
By the same lemma, the infinite intersection of the lines of $\lines$ is contained in a single point of the real part of $\DD^2$.
That is, this intersection is a line in the imaginary plane associated with a single point $p$ in the real part of $\DD^2$.
We say that $p$ is the \emph{special point} of the line family $\lines$.

Let $\ell \subset \RR^2$ be a line in the real part of $\DD^2$.
Then there could be several line families whose real part is $\ell$.
In addition, a line in $\DD^2$ whose real part is $\ell$ can participate in many line families that that have $\ell$ as their real line.
For example, the line defined by $y=(1+2\eps)x+(3+4\eps)$ is contained in the real line $y=x+3$ and is part of every family defined by $a_1=1, b_1=3,$ and $b_2 = x_1(m-a_2)$ such that $4= x_1(m-2)$ (see Lemma \ref{le:LineFamilyDual}(b)).
We now study the interaction between line families that have the same real line.

\begin{lemma} \label{le:LineFamProp}
Let $\lines_1$ and $\lines_2$ be two distinct line families in $\DD^2$ that correspond to the same real line $\ell$.
Assume that $\ell$ is not parallel to the $y$-axis. \\
(a) If $\lines_1$ and $\lines_2$ have the same special point then they have no lines in common. \\
(b) If $\lines_1$ and $\lines_2$ have different special points then they have at most one line in common.
\end{lemma}
\begin{proof}
As before, we define a line in $\DD^2$ using the equation $y=(a_1+a_2\eps)x + (b_1+b_2\eps)$.
Since the line families $\lines_1$ and $\lines_2$ have the same real part, every line in these families have the same values of $a_1$ and $b_1$.

(a) Denote the common special point as $(x_1,y_1)\in\RR^2$.
As shown in the proof of Lemma \ref{le:LineFamilyDual}(b), every two lines $y=ax+b$ and $y=a'x+b'$ from the same family satisfy a relation of the form $b'_2-b_2 = x_1(a_2-a'_2)$.

First assume that $x_1=0$.
In this case, every line of $\lines_1$ has the same $b_2$, and so does every line of $\lines_2$.
The two values of $b_2$ are distinct, since otherwise $\lines_1$ and $\lines_2$ would have been the same family.
Since no line can have two different values of $b_2$, the two line families are disjoint.

We now assume that $x_1\neq 0$.
Then exist $m_1,m_2\in \RR$ such that every line of $\lines_1$ satisfies $b_2=x_1(m_1-a_2)$ and every line of $\lines_2$ satisfies $b_2=x_1(m_2-a_2)$.
If a line satisfies both requirements, we obtain that $x_1(m_1-a_2)=x_1(m_2-a_2)$ and thus $m_1=m_2$.
This is impossible, since it implies that the two line families are identical.
We conclude that no line can be in both families.

(b) Denote the special point of $\lines_1$ as $(x_1,y_1)\in \RR^2$ and the special point of $\lines_2$ as $(x'_1,y'_1)\in \RR^2$.
Since these are distinct points on the line $\ell$ that is not parallel to the $y$-axis, we have $x_1 \neq x'_1$.
A line that is in both $\lines_1$ and $\lines_2$ satisfies $b_2=x_1(m_1-a_2)$ and $b_2=x'_1(m_2-a_2)$.
Since $x_1 \neq x'_1$ this system has at most one solution for the values of $a_2$ and $b_2$, implying that at most one line is in both families.
\end{proof}

\subsection{Adapting Elekes' argument to dual numbers} \label{ssec:ElekDual}

We are now ready to present our main proof for dual numbers.
We first repeat the relevant part of Theorem \ref{th:DualSP}

\begin{theorem} \label{th:DualMainCase}
Let $A$ be a set of $n$ dual numbers and multiplicity $n^{\alpha}$, for some $0 \le \alpha < 1/2$.
Then for any $\rho>0$,
\[ \max\{|A+A|,|AA|\} = \begin{cases}
\Omega^*\left(n^{3/2-5\alpha/8}n\right), & \quad 1/3\le\alpha<1/2,\\
\Omega\left(n^{5/4-\rho}\right), & \quad 0\le \alpha<1/3.
\end{cases} \]
\end{theorem}
\begin{proof}
By the multiplicity assumption, $A$ contains at most $n^\alpha$ non-invertible elements.
We discard these elements.
This does not change the asymptotic size of $A$ and can only decrease the sizes of $A+A$ and $AA$.
Thus, it suffices to prove the bound for the resulting smaller set.
Abusing notation, in the rest of the proof we refer to this revised set as $A$.

Consider the point set
\[ \pts = \{ (c,d)\in\DD^2 :\, c\in A+A \quad \text{ and } \quad d\in AA \}, \]
and the set of lines
\[ \lines = \{ y=c(x-d) :\, c,d\in A \}. \]
Note that $|\lines| =n^2$ and $|\pts| = |A+A|\cdot|AA|$.
Since the revised $A$ consists of invertible elements, there are no degenerate lines in $\lines$.

The proof is based on double counting $I(\pts,\lines)$.
A line of $\lines$ defined by $y=c(x-d)$ contains every point of $\pts$ of the form $(d+b,cb)$ for every $b\in A$.
That is, we have that $I(\pts,\lines) \ge |\lines||A| = n^3$.
For the rest of the proof we will derive upper bounds on $I(\pts,\lines)$.

We partition the incidences in $\pts\times\lines$ into two types, as follows.
We say that an incidence $(p,\ell)\in \pts\times\lines$ is \emph{special} if $p$ is incident to a second line $\ell'\in \lines$ such that $\ell$ and $\ell'$ are members of the same line family.
Since lines from the same family intersect only in their special point, the special point of this family is $(\re(p_x),\re(p_y))\in \RR^2$.
If an incidence $(p,\ell)\in \pts\times\lines$ is not special, we say that it is a \emph{standard} incidence.

We first bound the number of standard incidences.
By considering $\DD^2$ as $\RR^4$, we obtain an incidence problem with two-dimensional flats.
If $(p,\ell_1)$ and $(p,\ell_2)$ are standard incidences, then $\ell_1\cap \ell_2 = \{p\}$.
These are regular incidences, as defined before Theorem \ref{th:IncR4}.
By that theorem, the number of standard incidences is $O\left(|\pts|^{2/3+\rho}|\lines|^{2/3}+|\pts|+|\lines|\right)$.

When the number of standard incidences is larger than the number of special incidences, we have that
\[ I(\pts,\lines) = O\left(|\pts|^{2/3+\rho}|\lines|^{2/3}+|\pts|+|\lines|\right) = O\left(|A+A|^{2/3+\rho}|AA|^{2/3+\rho}n^{4/3}+ |A+A||AA|\right). \]

Combining this with $I(\pts,\lines) \ge n^3$ leads to $|A+A||AA| = \Omega(n^{5/2-\rho})$.
This immediately implies the assertion of the theorem, for any $0\le \alpha <1/2$.

\parag{Handling special incidences.}
It remains to consider the case where the number of special incidences is larger than the number of standard incidences.
Denote by $I(\alpha',\beta,\gamma,\delta)$ the number of special incidences $(p,\ell)\in \pts\times \lines$ that satisfy:
\begin{itemize}
\item Let $\ell_\RR$ be the line in the real part of $\DD^2$ that corresponds to $\ell$. Then $\ell_\RR$ corresponds to at least $n^{2\alpha'}$ lines of $\lines$ and to fewer than $2n^{2\alpha'}$ such lines.
\item There is a line family that contains $\ell$ whose special point is $(\re(p_x),\re(p_y))$, and that contains at least $n^\beta$ and fewer than $2n^\beta$ lines of $\lines$.
\item The real point $(\re(p_x),\re(p_y))$ corresponds to at least $n^\gamma$ points of $\pts$ and fewer than $2n^\gamma$ such points.
\item The real point $(\re(p_x),\re(p_y))$ is the special point of at least $n^\delta$ and to fewer than $2n^\delta$ line families that satisfy the property stated in the second item.
\end{itemize}

Note that we can take $\Theta(\log^4 n)$ elements $I(\alpha',\beta,\gamma,\delta)$ such that every special incidence in $\pts\times\lines$ is counted in at least one of those elements.
Thus, the number of special incidences is upper bounded by the maximum size of $I(\alpha',\beta,\gamma,\delta)$ times $\Theta(\log^4 n)$.

We study some basic properties of the parameters $\alpha',\beta,\gamma,\delta$ that maximize $I(\alpha',\beta,\gamma,\delta)$.
For this purpose, we assume that $\alpha',\beta,\gamma,\delta$ are fixed.
We denote by $S$ the set of special points that participate in incidences of $I(\alpha',\beta,\gamma,\delta)$.
Let $T$ be the set of lines in the real part of $\DD^2$ that correspond to lines of $\lines$ that participate in incidences of $I(\alpha',\beta,\gamma,\delta)$.
Let $F$ be the set of line families that contain at least $n^\beta$ lines of $\lines$ and fewer than $2n^{\beta}$ such lines.
Since $|\lines|=n^2$ and every line of $T$ corresponds to $\Theta(n^{2\alpha'})$ lines of $\lines$, we get that $|T| = O(n^{2-2\alpha'})$.

By the multiplicity of $A$ and the definition of $\lines$, at most $n^{2\alpha}$ lines of $\lines$ can correspond to the same line of $T$.
That is, we have $0\le \alpha'\le \alpha$.
We also have that $\beta \le 2\alpha'$, since otherwise there are not enough lines corresponding to a real line to create a family in $F$.

We consider the maximum number of lines from $\lines$ that a line family can contain.
Recall that a line in $\DD^2$ is defined by an equation of the form $y=(a_1+\eps a_2) x + (b_1+\eps b_2)$, and that a line in $\lines$ is defined by an equation of the form $y=c(x-d)$ with $c,d\in A$.
By Lemma \ref{le:LineFamilyDual}(b), all the lines in the same family have the same $a_1$ and $b_1$ values, so the real parts of $c$ and $d$ are fixed.
By the same lemma, either all the lines in a family have the same $b_2$ value, or they all satisfy a relation of the form $b_2 = x_1(m-a_2)$.
In either case, choosing the imaginary part of $c$ uniquely determines the imaginary part of $d$.
Due to the multiplicity of $A$, the family has at most $n^\alpha$ lines from $\lines$.
Since any line family contains at most $n^\alpha$ lines of $\lines$, we have that $0 \le \beta \le \alpha$.

Since the multiplicity of $A$ is $n^\alpha$, at most $n^{1+\alpha}$ sums in $A+A$ can have the same real part.
Similarly, at most $n^{1+\alpha}$ products in $AA$ can have the same real part.
Since $\pts= (A+A)\times (AA)$, at most $n^{2+2\alpha}$ points of $\pts$ can correspond to the same point in the real part of $\DD^2$.
That is, $0 \le \gamma \le 2+2\alpha$.
We also have the straightforward bound $n^\gamma \le |\pts|$, or equivalently $\gamma \le (\log |\pts|)/\log n$.

Next, we consider the maximum number of line families of $F$ that can have the same special point.
Recall that the lines of $\lines$ are defined as $y=c(x-d)$ where $c,d\in A$.
For every choice of $c$ and $s\in S$, there is a unique real part of $d$ such that the real part of the resulting line is incident to $s$.
That is, for a fixed special point and $c\in A$, there are at most $n^\alpha$ elements $d\in A$ such that the resulting line is incident to the special point.
By Lemma \ref{le:LineFamProp}(a), if two families have the same real line and the same special point, then they have no lines in common.
This yields $0 \le \delta \le 1+\alpha-\beta$.

To recap:
\begin{align*}
0\le \alpha', &\beta \le \alpha, \qquad \beta \le 2\alpha', \qquad 0 \le \delta \le 1+\alpha-\beta, \\[2mm]
0 &\le \gamma \le \min\left\{2+2\alpha, (\log |\pts|)/\log n\right\}.
\end{align*}

We next bound the number of families in $F$.
Recall that $|T|=O(n^{2-2\alpha'})$, and that each line of $T$ corresponding to fewer than $2n^{2\alpha'}$ lines of $\lines$.
For a fixed line $\ell\in T$, by Lemma \ref{le:LineFamProp} every two families corresponding to $\ell$ have at most one line in common.
There are fewer than $\binom{2n^{2\alpha'}}{2} = \Theta(n^{4\alpha'})$ pairs of lines of $\lines$ that correspond to $\ell$.
Each such pair can appear in at most one line family, and each line family subsumes at least $\binom{n^{\beta}}{2} = \Theta(n^{2\beta})$ such pairs.
Thus, the number of families that correspond to $\ell$ is $O(n^{4\alpha'-2\beta})$.
By summing up over every $\ell\in T$, we obtain that $|F| = O(n^{2+2\alpha'-2\beta})$.

We derive several upper bounds for $|S|$:
\begin{itemize}
\item Since each special point corresponds to $\Theta(n^\gamma)$ points of $\pts$, we have $|S| = O(|\pts|/n^\gamma)$.
\item Since $|F| = O(n^{2+2\alpha'-2\beta})$, and each special point subsumes $\Theta(n^{\delta})$ families of $F$, we obtain $|S| = O(n^{2+2\alpha'-2\beta-\delta})$.
\item Given a point $s\in S$, by Lemma \ref{le:LineFamProp}(a) each line of $T$ corresponds to $O(n^{2\alpha'-\beta})$ families that have $s$ as their special point. Thus, every point of $S$ is incident to $\Omega(n^{\delta-2\alpha'+\beta})$ lines of $T$. Recalling that $|T|=O(n^{2-2\alpha'})$, Theorem \ref{th:DualFormST} implies that
\[ |S|=O\left(\frac{(n^{2-2\alpha'})^2}{(n^{\delta-2\alpha'+\beta})^3} + \frac{n^{2-2\alpha'}}{n^{\delta-2\alpha'+\beta}}\right) = O\left(n^{4+2\alpha'-3\delta-3\beta} + n^{2-\delta-\beta}\right).\]
\end{itemize}

Consider the imaginary plane associated with a special point $s\in S$.
There are $\Theta(n^\delta)$ families incident to $s$, each corresponding to a distinct line in the imaginary plane of $s$.
There are $\Theta(n^\gamma)$ points of $\pts$ in this imaginary plane.
By the Szemer\'edi--Trotter theorem (Theorem \ref{th:ST83}), the number of incidences between these points and lines is $O\left(n^{2(\delta+\gamma)/3}+n^\delta+n^\gamma\right)$.
Since each line in the imaginary plane corresponds to $\Theta(n^\beta)$ lines of $\lines$, the number of incidences in the special point $s$ is
\begin{equation} \label{eq:IncInSpecialPnt}
O\left(n^\beta\left(n^{2(\delta+\gamma)/3}+n^\delta+n^\gamma\right)\right).
\end{equation}

To obtain an upper bound for $I(\alpha',\beta,\gamma,\delta)$, we can multiply \eqref{eq:IncInSpecialPnt} with any of our three upper bounds for $|S|$.
Then, to obtain an upper bound on the total number of incidences, we can multiply the resulting bound for $I(\alpha',\beta,\gamma,\delta)$ with $\Theta(\log^4 n)$.
We divide the rest of the analysis into cases, according to the term that dominates the inner parentheses in \eqref{eq:IncInSpecialPnt}.

\parag{The case where $n^{2(\delta+\gamma)/3}$ dominates.}
We first assume that $n^{2(\delta+\gamma)/3}$ is larger than the other two terms in the inner parentheses of \eqref{eq:IncInSpecialPnt}.
This case occurs when $\gamma/2 \le \delta \le \gamma$.
Using the bound $|S| = O(n^{2+2\alpha'-2\beta-\delta})$, we obtain that the number of special incidences is
\[ O\left(n^\beta\cdot n^{2(\delta+\gamma)/3} \cdot n^{2+2\alpha'-2\beta-\delta}\cdot \log^4 n\right) = O^*\left(n^{2+2\alpha'-\beta-\delta/3 + 2\gamma/3}\right)\]

Since we assume that the number of special incidences is larger than the number of standard incidences, the above is also a bound for the total number of incidences.
Combining this bound with $I(\pts,\lines) \ge n^3$ gives
\[ 2+2\alpha'-\beta-\delta/3 + 2\gamma/3 \ge 3, \quad \text{ or equivalently } \quad 2\alpha'-\beta-\delta/3 + 2\gamma/3 \ge 1. \]

Multiplying both sides by 3 and rearranging gives
\begin{equation} \label{eq:Case1Boots}
3-6\alpha'+3\beta+\delta - 2\gamma \le 0.
\end{equation}

We repeat the above argument with the different bound $|S|= O\left(n^{4+2\alpha'-3\delta-3\beta} + n^{2-\delta-\beta}\right)$.
In this case we get that the number of incidences is
\begin{align}
&O\left(n^\beta\cdot n^{2(\delta+\gamma)/3} \left(n^{4+2\alpha'-3\delta-3\beta} + n^{2-\delta-\beta}\right)\cdot\log^4n\right) \nonumber \\[2mm]
&\hspace{58mm}= O^*\left( n^{4+2\alpha'-7\delta/3-2\beta + 2\gamma/3} + n^{2-\delta/3+2\gamma/3}\right). \label{eq:TwoCases}
\end{align}

We split the current case into two additional cases, according to the dominating term in the above bound.

(i) When \eqref{eq:TwoCases} is dominated by the first term, combining it with $I(\pts,\lines) \ge n^3$ gives
\[ 4+2\alpha'-7\delta/3-2\beta + 2\gamma/3 \ge 3, \quad \text{ or equivalently } \quad 3+6\alpha'-7\delta-6\beta + 2\gamma \ge 0.     \]

Combining this with \eqref{eq:Case1Boots} gives
\[ 0\ge (3-6\alpha'+3\beta+\delta - 2\gamma) - (3+6\alpha'-7\delta-6\beta + 2\gamma) = -12\alpha'+9\beta+8\delta-4\gamma.\]
Dividing by 12 gives $0\ge -\alpha'+3\beta/4+2\delta/3-\gamma/3$.

We next use the bound $|S| = O(|\pts|/n^\gamma)$ to obtain that the number of incidences is
\begin{equation} \label{eq:SpecialIncPts}
O\left(n^\beta\cdot n^{2(\delta+\gamma)/3} \cdot |\pts|n^{-\gamma} \cdot \log^4n\right) = O^*\left( n^{\beta+2\delta/3-\gamma/3} \cdot |\pts|\right).
\end{equation}

Combining this with $I(\pts,\lines) \ge n^3$, and then applying $0\ge -\alpha'+3\beta/4+2\delta/3-\gamma/3$ and $\alpha',\beta\le \alpha$ yields
\begin{align*}
|A+A|\cdot |AA| = |\pts| &= \Omega^*\left(n^{3-(\beta+2\delta/3-\gamma/3)}\right) \\[2mm]
&= \Omega^*\left(n^{3-(\beta+2\delta/3-\gamma/3) + (-\alpha'+3\beta/4+2\delta/3-\gamma/3)}\right) \\[2mm]
&= \Omega^*\left(n^{3-\beta/4 -\alpha'}\right) = \Omega^*\left(n^{3-5\alpha/4}\right).
\end{align*}
This immediately implies $\max\{|A+A|,|AA|\} = \Omega^*\left(n^{3/2-5\alpha/8}\right)$.

(ii) We next consider the case where the incidence bound \eqref{eq:TwoCases} is dominated by the term $n^{2-\delta/3+2\gamma/3}$.
Combining this with $I(\pts,\lines) \ge n^3$ gives
\[ 2-\delta/3+2\gamma/3 \ge 3, \quad \text{ or equivalently } \quad 1/2+\delta/6-\gamma/3\le 0.\]

In this case we still have the bound \eqref{eq:SpecialIncPts} for the number of incidences.
Combining \eqref{eq:SpecialIncPts} with $I(\pts,\lines) \ge n^3$, and then applying $1/2+\delta/6-\gamma/3\le 0$, $\delta\le 1+\alpha-\beta$, and $\beta\le \alpha$ yields
\begin{align*}
|A+A|\cdot |AA| = |\pts| &= \Omega^*\left(n^{3-(\beta+2\delta/3-\gamma/3)}\right) \\[2mm]
&= \Omega^*\left(n^{3-(\beta+2\delta/3-\gamma/3) + (1/2+\delta/6-\gamma/3)}\right) = \Omega^*\left(n^{7/2-\beta -\delta/2}\right)\\[2mm]
&= \Omega^*\left(n^{7/2-\beta -(1+\alpha-\beta)/2}\log^{-4}n\right) =\Omega^*\left(n^{3-\alpha}\right).
\end{align*}
Similarly to the previous case, this implies $\max\{|A+A|,|AA|\} = \Omega^*\left(n^{3/2-\alpha/2}\right)$.

\parag{The cases where $n^{\delta}$ or $n^{\gamma}$ dominate.}
Assume that $n^\delta$ is larger than the other two terms in the inner parentheses of \eqref{eq:IncInSpecialPnt}.
This happens when $\delta>2\gamma$.
Using the bound $|S| = O(n^{2+2\alpha'-2\beta-\delta})$, we get that the number of incidences is
\[ O\left(n^{2+2\alpha'-2\beta-\delta} \cdot \left(n^\beta\cdot n^\delta\right)\cdot \log^{4}n\right) = O^*\left(n^{2+2\alpha'-\beta} \right). \]

Combining this with $I(\pts,\lines) \ge n^3$ implies that $2+2\alpha'-\beta \ge 3$, or equivalently $\alpha' \ge (1+\beta)/2$.
This in turn implies that $\alpha \ge \alpha' \ge (1+\beta)/2 \ge 1/2$.
Since this contradicts the assumption concerning $\alpha$, we conclude that $n^\delta$ cannot dominate the inner parentheses of \eqref{eq:IncInSpecialPnt}.

Finally, assume that $n^\gamma$ is larger than the other two terms in the inner parentheses of \eqref{eq:IncInSpecialPnt}.
This happens when $\delta>2\gamma$.
By using the bound $|S| = O(|\pts|/n^\gamma)$ we get that the number of incidences is
\[ O\left(|\pts|n^{-\gamma} \cdot \left(n^\beta\cdot n^\gamma\right)\cdot \log^{4}n\right) = O^*\left(|\pts|n^\beta\right). \]

Combining this with $I(\pts,\lines) \ge n^3$ implies
\[ |A+A|\cdot|AA| = |\pts| = \Omega^*(n^{3-\beta}) = \Omega^*(n^{3-\alpha}).\]
This immediately implies $\max\{|A+A|,|AA|\} = \Omega^*\left(n^{3/2-\alpha/2}\right)$.

By going over each case that occurs when the number special incidences is larger, we note that the weakest bound that was obtained is $\max\{|A+A|,|AA|\} = \Omega^*\left(n^{3/2-5\alpha/8}\right)$.
To complete the proof, for each value of $\alpha$ we use the weaker bound out of the one obtained when there are more standard incidences, and the one obtained when there are more special incidences.
\end{proof}

\parag{Remark.} It may at first seem surprising that in our analysis of special incidences we obtain bounds such as $\max\{|A+A|,|AA|\} = \Omega^*\left(n^{3/2-5\alpha/8}\right)$.
In particular, when $\alpha=0$ there is no multiplicity and each family consists of a single line, so one might expect to get the standard Elekes bound of $\Omega(n^{5/4})$.
The reason for obtaining a stronger bound is our assumption that each family has a single special point.
Thus, when setting $\alpha=0$, we force each line to form an incidence with at most one point.
It is not surprising that we get a stronger bound under such a strong assumption.
\vspace{2mm}

We next prove the bound of Theorem \ref{th:DualSP} for the case where $1/2 \le \alpha < \kappa$.
Recall that $\kappa = (39 - \sqrt{721})/20 \approx 0.607$.

\begin{corollary} \label{co:DualLargeAlpha}
Let $A$ be a set of $n$ dual numbers with multiplicity $n^\alpha$, for some $1/2\le \alpha<\kappa$.
Then for any $\rho>0$,
\[ \max \left\{|A+A|,|AA|\right\} = \Omega\left(n^{9/4-39\alpha/16+5\alpha^2/8-\rho}\right). \]
\end{corollary}
\begin{proof}
Consider a sufficiently small $\rho'>0$.
We remove elements from $A$ until it has multiplicity $n^{1/2-\rho'}$.
This yields a subset $A'\subset A$ of size $\Omega\left(n^{1-(\alpha+\rho'-1/2)}\right) = \Omega\left(n^{3/2-\alpha-\rho'}\right)$.
Applying Theorem \ref{th:DualSP} on $A'$ with multiplicity $1/2-\rho'$, and assuming that $\rho'$ is sufficiently small, leads to
\begin{align*}
\max \left\{|A+A|,|AA|\right\} &\ge \max \left\{|A'+A'|,|A'A'|\right\} = \Omega^*\left(|A'|^{3/2-5\alpha/8}\right) \\[2mm]
&= \Omega^*\left(n^{(3/2-5\alpha/8)(3/2-\alpha-\rho')}\right) = \Omega\left(n^{9/4-39\alpha/16+5\alpha^2/8-\rho}\right).
\end{align*}
Finally, $9/4-39\alpha/16+5\alpha^2/8>1$ when $\alpha<\kappa$.
\end{proof}

\subsection{Adapting Solymosi's argument to dual numbers} \label{ssec:SolyDual}

For any $a,a'\in \DD$ we have $\re(a \cdot a') = \re(a) \cdot \re(a')$.
When $a'$ is invertible, we also have $\re(a/a') = \re(a)/\re(a')$.
For $\lambda \in \RR$, we define
\begin{align*}
r_{A}^{\times}(\lambda)&=\left|\left\{ (a,a')\in A^{2}\ :\ \re(a\cdot a')=\lambda\right\} \right|, \\[2mm]
r_{A}^{\div}(\lambda)&=\left|\left\{ (a,a')\in A^{2}\ :\ \re(a/a')=\lambda\right\} \right|.
\end{align*}

In other words, $r_{A}^{\times}(\lambda)$ is the number of ways to obtain $\lambda$ as the real part of a product of two elements of $A$, and similarly for $r_{A}^{\div}(\lambda)$.
For a finite set $A \subset \DD$, we define the \emph{multiplicative energy} of $A$ as
\[ E^\times (A) = \left|\left\{ (a,b,c,d)\in A^4\ :\ a\cdot b = c \cdot d \right\} \right|.\]

We are now ready to adapt Solymosi's sum-product argument \cite{Soly09} to sets of dual numbers.

\begin{theorem} \label{th:SolyDual}
Let $A$ be a set of $n$ dual numbers with multiplicity $n^{\alpha}$, for some $0\le \alpha<1/2$.
Then
\[ \max\{|A+A|,|AA|\}=\Omega^*\left(n^{(4-2\alpha)/3}\right). \]
\end{theorem}
\begin{proof}
By assumption, $A$ may contain up to $n^\alpha$ non-invertible elements.
We discard these elements without changing the asymptotic size of $|A|$.

If at least half of the elements of $A$ have a positive real part, we discard from $A$ elements with a negative real part.
Otherwise, we discard from $A$ the elements that have a positive real part and multiply the remaining elements by $-1$.
In either case, all the elements of the revised set have a positive real part.
The asymptotic size of $A$ is unchanged and the sizes of $A+A$ and $AA$ can only decrease.
Thus, it suffices to derive a lower bound for $\max\{|A+A|,|AA|\}$ for the revised $A$.
Abusing notation, we still refer to this set as $A$ and its size as $n$.

Since each pair $(a,a')\in A^{2}$ contributes to exactly one set $r_{A}^{\div}(\lambda)$, we have %
\[ \sum_{\lambda\in\re(A/A)}r_{A}^{\div}(\lambda)=n^{2}. \]

If $(a_{1},a_{2},a_{3},a_{4})\in A^{4}$ satisfies $a_{1}a_{2}=a_{3}a_{4},$ then $\re(a_{1}/a_{3})=\re(a_{4}/a_{2})$.
This implies that
\[ E^{\times}(A)\le\sum_{\lambda\in\re(A/A)}r_{A}^{\div}(\lambda)^{2}. \]

Using dyadic decomposition, we partition this sum to
\[ E^{\times}(A)\le\sum_{m=0}^{\log n-1}\sum_{\substack{\lambda\in\re(A/A)\\
2^{m}\le\r{\lambda}<2^{m+1}}}r_{A}^{\div}(\lambda)^{2}. \]

This implies that there exists $0\le m<\log n$ such that
\[ \sum_{\substack{\lambda\in\re(A/A)\\ 2^m\le\r{\lambda}<2^{m+1}}}r_{A}^{\div}(\lambda)^{2}\ge\frac{E^{\times}(A)}{\log n}. \]

We set $\Lambda=\left\{ \lambda\in\re(A/A)\ :\ 2^m\le\r \lambda<2^{m+1}\right\}$, and denote the elements of $\Lambda$ as  $0<\lambda_{1}<\lambda_{2}<\cdots<\lambda_{|\Lambda|}$.
Since $r_{A}^{\div}(\lambda)^{2}<2^{2m+2},$ we have
\begin{equation} \label{eq:MultEnergyDecomp}
1>\frac{E^{\times}(A)}{|\Lambda|2^{2m+2}\log n}.
\end{equation}

Consider the planar point set ${\cal P}=A\times A\subset\DD^{2}.$
Since ${\cal P}+{\cal P}=(A+A)\times(A+A)$, we have that $|{\cal P}+{\cal P}|=|A+A|^{2}$.

For each $1\le i\le|\Lambda|,$ let $\ell_{i}$ denote the line in $\RR^{2}$ defined by $y=\lambda_{i}x$.
We think of these lines as being in the real part of $\DD^2$ (which is a copy of $\RR^2$).
Let ${\cal P}\cap\ell_{i}$ be the set of points $(a,b)\in {\cal P}$ that satisfy $\re(a)=\lambda_{i} \cdot \re(b)$.
In other words, this is the set of points that satisfy the real part of the line equation, but not necessarily the imaginary part.
By definition, for each of these $|\Lambda|$ lines we have $2^m \le |{\cal P}\cap \ell_{i}|<2^{m+1}$.
Let ${\cal P}\cap_{\RR}\ell_{i}$ denote the set of points in the real part of $\DD^2$ that correspond to at least one point of ${\cal P}\cap\ell_{i}$.
Note that ${\cal P}\cap\ell_{i}$ is in $\DD^2$ while ${\cal P}\cap_{\RR}\ell_{i}$ is in $\RR^2$, and that $|{\cal P}\cap\ell_{i}|\ge |{\cal P}\cap_{\RR}\ell_{i}|$.

The lines $\ell_i \subset \RR^2$ are all incident to the origin.
In addition, the points of $({\cal P}\cap_{\RR}\ell_{i})+({\cal P}\cap_{\RR}\ell_{i+1})$ lie in the interior of the wedge formed by $\ell_{i}$ and $\ell_{i+1}$ in the first quadrant of $\RR^{2}$.
Thus, for any $i\ne i'$, the sets $({\cal P}\cap_{\RR}\ell_{i})+({\cal P}\cap_{\RR}\ell_{i+1})$ and $({\cal P}\cap_{\RR}\ell_{i'})+({\cal P}\cap_{\RR}\ell_{i'+1})$ are disjoint.

Fix $0 < i <|\Gamma|$.
For any $a_{1},a_{2}\in \ell_{i}$ and $a_{3},a_{4}\in \ell_{i+1}$ (these are points in $\RR^2$),
we have that $a_{1}+a_{3}\ne a_{2}+a_{4}$ unless $a_{1}=a_{2}$ and $a_{3}=a_{4}$.
Indeed, for variables $c,d\in \RR$, the system $(c,c\cdot\lambda_{i})+(d,d\cdot\lambda_{i+1})=(p_{x},p_{y})$ has a unique solution.
Hence, for any $p,q\in{\cal P}\cap \ell_{i}$ and $r,s\in{\cal P}\cap \ell_{i+1}$ that satisfy $\re (p)\neq \re (q)$ or $\re (r) \neq \re (s)$, we have $p+r\ne q+s$.
Since ${\cal P} = A\times A$ and since $A$ has multiplicity $n^\alpha$, for each $(s,t)\in \RR^2$ at most $n^{2\alpha}$ pairs $(a,b)\in \pts$ satisfy $(\re (a),\re (b)) = (s,t)$.
For each point in ${\cal P}\cap_{\RR}\ell_{i+1}$ we arbitrarily consider one point of ${\cal P}\cap\ell_{i+1}$ that corresponds to it, and denote the resulting set as $S_i$.
Note that $|S_i|\ge |{\cal P}\cap\ell_{i+1}|/n^{2\alpha}$ and that $S_i$ consists of points with distinct real parts.
We claim that $|({\cal P}\cap \ell_{i})+ S_i| = |{\cal P}\cap\ell_{i}| \cdot |S_i|$.
In other words, we claim that every element of $({\cal P}\cap \ell_{i})+ S_i$ can be written as a sum in a unique way.
Indeed, for $s\in S_i$ and $a,a'\in {\cal P}\cap \ell_{i}$ we clearly have $a+s \neq a'+s$ when $\re(a)\neq \re(a')$.
If $\re(a)= \re(a')$ then $a$ and $a'$ have distinct imaginary parts, again implying $a+s \neq a'+s$.
This leads to
\[ \left|({\cal P}\cap \ell_{i})+({\cal P}\cap \ell_{i+1})\right| \ge \left|({\cal P}\cap \ell_{i})+S_i\right| \ge |{\cal P}\cap \ell_{i}|\cdot |{\cal P}\cap \ell_{i+1}|/n^{2\alpha}. \]

Combining this with \eqref{eq:MultEnergyDecomp} yields
\begin{align}
|A+A|^{2} & =|{\cal P}+{\cal P}| >\sum_{i=1}^{|\Lambda|-1}\left|({\cal P}\cap \ell_{i})+({\cal P}\cap\ell_{i+1})\right|  \ge\sum_{i=1}^{|\Lambda|-1}\frac{|{\cal P}\cap\ell_{i}||{\cal P}\cap\ell_{i+1}|}{n^{2\alpha}} \nonumber \\[2mm]
 & \ge\frac{(|\Lambda|-1)2^{2m}}{n^{2\alpha}} \ge\frac{(|\Lambda|-1)2^{2m}}{n^{2\alpha}}\cdot \frac{E^{\times}(A)}{|\Lambda|2^{2m+2}\log n}  = \Omega^*\left(\frac{E^{\times}(A)}{n^{2\alpha}}\right). \label{eq:MultEnergyUpper}
\end{align}

By the Cauchy-Schwarz inequality,
\begin{align*}
E^{\times}(A)  =\sum_{t\in AA}r_{A}^{\times}(t)^{2}  \ge\frac{\left(\sum_{t\in AA}r_{A}^{\times}(t)\right)^{2}}{|AA|} =\frac{n^{4}}{|AA|}.
\end{align*}

Combining this with \eqref{eq:MultEnergyUpper} leads to
\[ |A+A|^{2}n^{2\alpha} = \Omega^*\left(\frac{n^{4}}{|AA|}\right). \]
Rearranging this gives
\[ |A+A|^{2}|AA| =\Omega^*\left(n^{4-2\alpha}\right). \]
This immediately implies the assertion of the theorem.
\end{proof}

\section{Double numbers} \label{sec:Double}

In this section we study double numbers, and in particular prove Theorem \ref{th:DoubleSP}.
In Section \ref{ssec:DoubleLines} we study properties of lines in the double plane.
We derive a point-line incidence bound in $\SSS^2$, and study additional properties of such incidences.
This case is more involved than the analog for dual lines in Section \ref{sec:Dual}, since we cannot easily separate $\SSS^2$ into a real part and an imaginary part as we did for $\DD^2$.
In Section \ref{ssec:ElekDouble} we adapt Elekes' sum-product argument to the double numbers.
As in the dual case, we add several additional steps to Elekes' original approach.
In Section \ref{ssec:SolyDouble}, we adapt Solymosi's sum-product argument to the double numbers.

\subsection{Lines in the double plane} \label{ssec:DoubleLines}

Recall that we denote by $\SSS$ the set of the double numbers: The extension of $\RR$ with the extra element $j$ and the rule $j^2=1$.
We write a number $a\in \SSS$ as $a_1+j a_2$.
Multiplication of double numbers is commutative, and 1 is the unit element.
A double number $a\in \SSS$ has an inverse element if and only if $a_1\neq \pm a_2$ (equivalently, $a_1^2\neq a_2^2$).
The inverse element is then $a^{-1} = \frac{a_1-j a_2}{a_1^2-a_2^2}$.
Indeed,
\[ a\cdot a^{-1} = \frac{(a_1+j a_2)(a_1-j a_2)}{a_1^2-a_2^2} = \frac{a_1^2-a_2^2}{a_1^2-a_2^2} = 1.\]

For $a=a_1+ja_2 \in \SSS$, we define $\Delta^+(a) = a_1+a_2$ and $\Delta^-(a) = a_1-a_2$.
For any $a,b\in \SSS$ where $b$ is invertible, we have
\begin{align}
&\Delta^+(a + b) = \Delta^+(a_1 + b_1+(a_2 + b_2)j) = \Delta^+(a) + \Delta^+(b), \nonumber \\[2mm]
&\Delta^+(a \cdot b) = \Delta^+(a_1b_1+a_2b_2+(a_1b_2+ a_2b_1)j) = a_1b_1+a_2b_2 + a_1b_2+ a_2b_1 = \Delta^+(a)\cdot \Delta^+(b), \nonumber \\[2mm]
&\Delta^+(a/b) = \Delta^+\left(\frac{(a_1+a_2j)(b_1-b_2j)}{b_1^2-b_2^2}\right) \nonumber \\[2mm]
&\hspace{40mm}= \Delta^+\left(\frac{a_1b_1-a_2b_2+(a_2b_1- a_1b_2)j}{b_1^2-b_2^2}\right) = \Delta^+(a)\cdot \Delta^+(b^{-1}). \label{eq:DeltaProp}
\end{align}
It is not difficult to verify that the above equations still hold when replacing $\Delta^+(\cdot)$ with $\Delta^-(\cdot)$.

We define a line in $\SSS^2$ as the set of points on which a linear equation vanishes.
Let $\ell$ be the line defined by $ax+by=c$, where $a,b,c\in \SSS$.
This corresponds to
\begin{align*}
(a_1+j a_2) (x_1+j x_2) + (b_1+j b_2)(y_1+j y_2) &= (c_1+j c_2),
\end{align*}
or equivalently
\begin{align*}
a_1x_1 +a_2x_2  +  b_1y_1 +b_2y_2 &= c_1, \\
a_2x_1 + a_1 x_2 + b_2 y_1 + b_1 y_2 &= c_2.
\end{align*}

The two above equations are linearly dependent if and only if $a_1= \pm a_2$, $b_1= \pm b_2$, and $c_1 = \pm c_2$, where all three $\pm$ represent the same operation.
We can think of $\SSS^2$ as $\RR^4$, and then $\ell$ is either a 2-flat or a hyperplane, depending on whether the two above equations are linearly dependent.
We refer to the lines of the latter type as ``degenerate lines''.
Note that for a line defined by $ax+by=c$ to degenerate, all three $a$, $b$, and $c$ must be non-invertible.

As in the dual case, lines in the double plane can have an infinite intersection, even when excluding non-invertible coefficients in the line equations.
For example, consider the set of non-degenerate lines
\[ \lines = \{ y=(k+(k-1)j)x+((15-3k)+(9-3k)j) :\ k\in \RR \}. \]

It is not difficult to verify that every line of $\lines$ contains the line $\ell$ parameterized by $(c+(3-c)j,12+c+(9-c)j)\in \SSS^2$ with $c\in \RR$.
By taking $n$ lines from $\lines$ and $m$ points on $\ell$, we get $mn$ incidences.
That is, the point--line incidence problem in $\SSS^2$ is trivial.
We now study when collections of lines have an infinite intersection.

\begin{lemma} \label{le:LineFamilyDoub}
In each of the following parts, every $\pm$ represents the same operation, and every $\mp$ represents the other operation. \\
(a) Let $\ell$ and $\ell'$ be distinct lines in $\SSS^2$, respectively defined by $y=ax+b$ and $y=a'x+b'$.
The intersection $\ell\cap\ell'$ contains more than one point if and only if $a_1-a'_1=\pm(a_2-a'_2)\neq 0$, and $b_1-b'_1 = \pm (b_2-b'_2)$.
When these conditions are satisfied, $\ell_1\cap\ell_2$ is a line in $\RR^4$ and every point $x$ in $\ell_1\cap\ell_2$  satisfies $x_1 \pm x_2 = \frac{b'_1-b_1}{a_1-a'_1}$.  \\
(b) Let $\lines$ be a set of lines of the form $y=ax+b$ that have an infinite common intersection.
Then exist $t_1,t_2\in \RR$ such that every line of $\lines$ satisfies $\Delta^\mp(a)=t_1$ and $\Delta^\mp(b)=t_2$.
There also exists $s$ such that every point $(x,y)$ in the infinite intersection satisfies $\Delta^\pm(x)=s$.
\begin{itemize}
\item If $s = 0$ then every line has the same $b$, and every point in the common intersection satisfies $\Delta^{\pm}(y) = \Delta^{\pm}(b)$.
\item If $s \neq 0$ then exist $m,m'\in \RR$ such that every line satisfies $b_1= s(m-a_1)$ and $b_2= s(m'-a_2)$.
Every point in the common intersection satisfies $\Delta^{\pm}(y) = s(m\pm m')$.
\end{itemize}
\end{lemma}
\begin{proof}
(a) To study the intersection points of the two lines, we combine $y=ax+b$ and $y=a'x+b'$, obtaining $ax+b=a'x+b'$, or equivalently $x(a-a')=b'-b$.
Splitting this equation into real and imaginary parts gives
\begin{align}
x_1(a_1-a'_1) + x_2(a_2-a'_2) &= b'_1-b_1,  \nonumber \\
x_1(a_2-a'_2) + x_2(a_1-a'_1) &= b'_2-b_2. \label{eq:TwoLinearEqDouble}
\end{align}

We consider the above as a linear system in $x_1$ and $x_2$.
This system has a unique solution unless $(a_1-a'_1)^2 = (a_2-a'_2)^2$.
That is, the intersection contains at most one point unless $a_1-a'_1=\pm(a_2-a'_2)$.
If $(a_1-a'_1)^2 = (a_2-a'_2)^2=0$ then the two lines are either parallel or identical.
Thus, it remains to study the case where  $a_1-a'_1=\pm(a_2-a'_2)\neq 0$.

We either have that $a_1-a'_1 = a_2-a'_2$ or that $a_1-a'_1 = a'_2-a_2$.
We first consider the former case.
By the equations of \eqref{eq:TwoLinearEqDouble}, either $b'_1-b_1 \neq b'_2-b_2$ and the two lines do not intersect or $b'_1-b_1 = b'_2-b_2$ and the two lines have an infinite intersection.
In the case of an infinite intersection, the equations of \eqref{eq:TwoLinearEqDouble} also imply $x_1 + x_2= \frac{b'_1-b_1}{a_1-a'_1}$.

It remains to consider the case where $a_1-a'_1 = a'_2-a_2 \neq 0$.
By \eqref{eq:TwoLinearEqDouble}, either $b'_1-b_1 \neq b_2-b'_2$ and the two lines do not intersect or $b'_1-b_1 = b_2-b'_2$ and the two lines have an infinite intersection.
In the case of an infinite intersection, the equations of \eqref{eq:TwoLinearEqDouble} also imply $x_1 - x_2= \frac{b'_1-b_1}{a_1-a'_1}$.

(b) If distinct 2-flats in $\RR^4$ have an infinite intersection, then this intersection is a line.
Let $\ell^*$ be the line in $\RR^4$ that is the infinite intersection of the lines of $\lines$.
By part (a), there exists $s\in \RR$ such that every $(x,y)\in \ell^*$ satisfies $\Delta^{\pm}(x) = s$ and every distinct $\ell,\ell'\in \lines$ satisfy $\frac{b'_1-b_1}{a_1-a'_1}=s$.
This implies that the symbol $\pm$ represents the same operation for all lines in $\lines$.
By part (a) we also have that $a_1-a_1' = \pm (a_2-a'_2)$, or equivalently that $\Delta^{\mp}(a) = \Delta^{\mp}(a')$.
That is, every line of $\lines$ has the same value for $\Delta^{\mp}(a)$.
Similarly, the condition $b_1-b'_1 = \pm (b_2-b'_2)$ leads to every line of $\lines$ having the same value for $\Delta^{\mp}(b)$.

If $s=0$, then every two lines $\ell,\ell'\in \lines$ satisfy $\frac{b'_1-b_1}{a_1-a'_1}=0$, or equivalently $b_1=b'_1$.
Since every $\Delta^{\mp}(b)$ has the same value, we also obtain $b_2=b'_2$.
That is, every line of $\lines$ has the same $b$.
Consider a line of $\lines$ defined by $y=ax+b$.
Splitting this equation to real and imaginary parts, we obtain $y_1 = a_1x_1+a_2x_2 +b_1$ and $y_2 = a_1x_2+a_2x_1 +b_2$.
Combining these two equations gives
\[ y_1 \pm y_2 = a_1x_1+a_2x_2 +b_1 \pm (a_1x_2+a_2x_1 +b_2) = (a_1\pm a_2)(x_1\pm x_2) + b_1\pm b_2 = b_1 \pm b_2. \]

If $s\neq0$ then every distinct $\ell,\ell'\in \lines$ satisfy $b'_1-b_1=s(a_1-a'_1)$.
This implies that there exists $m\in \RR$ such that every line of $\lines$ satisfies $b_1= s(m-a_1)$.
By part (a), we have $\frac{b_1-b'_1}{a_1-a'_1} = \frac{b_2-b'_2}{a_2-a'_2}$, or equivalently $\frac{b'_1-b_1}{a_1-a'_1} = \frac{b'_2-b_2}{a_2-a'_2}$.
Thus, $b'_2-b_2=s(a_2-a'_2)$ and there exists $m'\in \RR$ such that every line of $\lines$ satisfies $b_2= s(m'-a_2)$.

Consider a line of $\lines$ defined by $y=ax+b$.
Splitting this equation into real and imaginary parts, we obtain $y_1 = a_1x_1+a_2x_2 +b_1$ and $y_2 = a_1x_2+a_2x_1 +b_2$.
Combining these two equations gives
\begin{align*}
y_1 \pm y_2 &= a_1x_1+a_2x_2 +b_1 \pm (a_1x_2+a_2x_1 +b_2) = (a_1 \pm a_2)(x_1\pm x_2) + (b_1\pm b_2) \\[2mm]
&= s(a_1 \pm a_2) + s(m-a_1 \pm (m'-a_2)) = s(m\pm m').
\end{align*}
\end{proof}

The example before Lemma \ref{le:LineFamilyDoub} was obtained by setting $x_1+x_2 = 3, m=5, m'=2,$ and $a_1-a_2 = 1$.
The rest followed from Lemma \ref{le:LineFamilyDoub}.

Using Lemma \ref{le:LineFamilyDoub}, we can prove Theorem \ref{th:DoubleST}.
This proof is identical to the proof of Theorem \ref{th:DualST}, so we do not repeat it here.

To derive our sum-product bounds in $\SSS$, we need additional properties of point-line incidences in $\SSS^2$.
We refer to a set of lines that have an infinite common intersection as a \emph{family of lines}.
Let $\lines$ be such a family.
By Lemma \ref{le:LineFamilyDoub}(b), there exist constants $s,s'\in \RR$ such that every point $(x,y)\in\SSS^2$ in the common intersection of the lines of $\lines$ satisfies $\Delta^{\pm}(x) = s$ and $\Delta^{\pm}(y) =s'$.
We say that $(s,s')\in \RR^2$ is the \emph{point parameter} of the family $\lines$.
Also by Lemma \ref{le:LineFamilyDoub}(b), there exist $t_1,t_2\in \RR$ such that every line of $\lines$ satisfies $\Delta^\mp(a)=t_1$ and $\Delta^\mp(b)=t_2$.
We define the \emph{line parameter} of $\lines$ to be $(t_1,t_2)$.

We can study the interaction between different line families by studying their point parameters and line parameters.
We say that a line family is \emph{positive} or \emph{negative} according to the meaning of the $\pm$ sign in the definition of the point parameter of the family.
In the notation of Lemma \ref{le:LineFamilyDoub}(b), a family is positive if $s=x_1+x_2$.
We refer to this property as the \emph{sign} of a line family.

\begin{lemma} \label{le:LineFamDouble}
Let $\lines_1$ and $\lines_2$ be distinct line families in $\SSS^2$ with the same sign. \\
(a) If $\lines_1$ and $\lines_2$ do not have the same line parameter, then no line is contained in both families. \\
(b) If $\lines_1$ and $\lines_2$ have the same line parameter and the same value for $\Delta^{\pm}(x)$, then no line is contained in both families. \\
(c) If $\lines_1$ and $\lines_2$ have the same line parameter but not the same $\Delta^{\pm}(x)$, then at most one line of $\SSS^2$ is in both families.
\end{lemma}
\begin{proof}
(a) By the definition of the line parameter, either the lines of $\lines_1$ and the lines of $\lines_2$ have different values of $\Delta^{\mp}(a)$ or these lines have different values of $\Delta^{\mp}(b)$ (or both).
Since no line can have two different values for $\Delta^{\mp}(a)$ or two different values for $\Delta^{\mp}(b)$, the two families are disjoint.

(b) Let $\pm$ denote the sign of $\lines_1$ and $\lines_2$, let $\mp$ denote the opposite sign, and denote the common value of $\Delta^{\pm}(x)$ as $s$.
Since the two families have the same line parameter and the same sign, every line in $\lines_1 \cup \lines_2$ has the same value of $\Delta^{\mp}(a)$ and the same value of $\Delta^{\mp}(b)$.
We assume for contradiction that there exists a line in $\SSS^2$ that is contained in both families.

We first consider the case of $s=0$.
By Lemma \ref{le:LineFamilyDoub}(b), all the lines in the same family have the same $b$.
Since the two families contain a line in common, every line of $\lines_1 \cup \lines_2$ has the same value of $b_1$ and the same value of $b_2$.
Since these families also have the same value of $\Delta^{\mp}(a)$ they are identical, contradicting the assumption.

Next, consider the case of $s\neq 0$.
By Lemma \ref{le:LineFamilyDoub}(b), in this case there exist $m_1,m_2\in \RR$ such that every line of $\lines_1$ satisfies $b_1 = s(m_1-a_1)$ and every line of $\lines_2$ satisfies $b_1 = s(m_2-a_1)$.
Since there is a line in both families, we obtain $s(m_1-a_1)=s(m_2-a_1)$ or equivalently $m_1=m_2$.
By a symmetric argument, the $m'$ values of both families are identical.
Since the value of $\Delta^{\mp}(a)$ is also identical for both line families, we conclude that the two families are identical, contradicting the assumption.

We got a contradiction in both cases, so the two line families cannot have any lines in common.

(c) Let $\pm$ denote the sign of $\lines_1$ and $\lines_2$, let $\mp$ denote the opposite sign, and let $\ell\subset \SSS^2$ be a line that is in both families.
As in the proof of part (b), every line in $\lines_1 \cup \lines_2$ has the same value of $\Delta^{\mp}(a)$ and the same value of $\Delta^{\mp}(b)$.
Denote the $\Delta^{\pm}(x)$ values of $\lines_1$ and $\lines_2$ as $s_1$ and $s_2$, respectively.
Then there exist $m_1,m_2$ such that $\ell$ satisfies $a_1 = s_1(m_1-b_1)$ and $a_1 = s_2(m_2-b_1)$.
These are two independent linear equations in the variables $a_1,b_1$, and thus have a unique solution.
We conclude that there is at most one line common to both families.
\end{proof}

We can also use the line parameter to study the behavior of line families in $\RR^4$.

\begin{lemma} \label{le:FamInHyper}
When considering every line in $\SSS^2$ as a 2-flat in $\RR^4$, the 2-flats of a line family are all contained in a common hyperplane.
Two line families of the same sign are contained in the same hyperplane if and only if they have the same line parameter.
\end{lemma}
\begin{proof}
Consider a line family $\lines$ with line parameter $(t,t')$, and a line from $\lines$ defined by $y=ax+b$.
Splitting this equation into real and imaginary parts, we obtain $y_1 = a_1x_1+a_2x_2 +b_1$ and $y_2 = a_1x_2+a_2x_1 +b_2$.
Combining these two equations leads to
\[ y_1\mp y_2 = (a_1\mp a_2)(x_1\mp x_2) + (b_1 \mp b_2) = t(x_1 \mp x_2)+t'.\]

That is, every line of $\lines$ corresponds to a 2-flat that is contained in the hyperplane defined by $y_1 \mp y_2=t(x_1 \mp x_2)+t'$.
It can now be easily verified that two families are contained in the same hyperplane if and only if they have the same line parameter $(t,t')$.
\end{proof}

Finally, we study the interaction between two line families with opposite signs.

\begin{lemma} \label{le:OppositeSigns}
Let $\lines_1$ and $\lines_2$ be distinct line families in $\SSS^2$ with opposite signs.
Then at most one line is contained in both families.
\end{lemma}
\begin{proof}
Without loss of generality, assume that the lines of $\lines_1$ have the same value of $\Delta^+(a)$ and that the lines of $\lines_2$ have the same value of $\Delta^-(a)$.
Then all the lines of $\lines_1$ have the same value of $\Delta^+(b)$ and all the lines of $\lines_2$ have the same value of $\Delta^-(b)$.
It can be easily verified that there are unique values for $a_1,a_2,b_1,b_2$ that satisfy all four restrictions.
We conclude that at most one line can be in both families.
\end{proof}

\subsection{Adapting Elekes' argument to double numbers} \label{ssec:ElekDouble}

We are now ready to adapt the proof from Section \ref{ssec:ElekDual} to the double numbers.
The two proofs are similar, but not identical.
We thus provide most of the proof, skipping only the last part, which is technical calculation identical to the one in Section \ref{ssec:ElekDual}.
We first repeat the relevant part of Theorem \ref{th:DoubleSP}.

\begin{theorem}
Let $A$ be a set of $n$ double numbers and multiplicity $n^{\alpha}$, for some $0 \le \alpha < 1/2$.
Then for any $\rho>0$,
\[ \max\{|A+A|,|AA|\} = \begin{cases}
\Omega^*\left(n^{3/2-3\alpha/4}\right), & \quad 1/3\le\alpha<1/2, \\
\Omega\left(n^{5/4-\rho}\right), & \quad 0\le \alpha<1/3.
\end{cases} \]
\end{theorem}
\begin{proof}
By the multiplicity assumption, $A$ contains at most $2n^\alpha$ non-invertible elements.
We discard these elements.
This does not change the asymptotic size of $A$ and can only decrease the sizes of $A+A$ and $AA$.
Thus, it suffices to prove the bound for the resulting smaller set.
Abusing notation, in the rest of the proof we refer to this revised set as $A$.

Consider the point set
\[ \pts = \{ (c,d)\in\DD^2 :\, c\in A+A \quad \text{ and } \quad d\in AA \}, \]
and the set of lines
\[ \lines = \{ y=c(x-d) :\, c,d\in A \}. \]
Note that $|\lines| =n^2$ and $|\pts| = |A+A|\cdot|AA|$.
Since the revised $A$ consists only of invertible elements, there are no degenerate lines in $\lines$.
We think of $\pts$ both as a point set in $\SSS^2$ and as a point set in $\RR^4$.
Similarly, we think of $\lines$ both as a set of lines in $\SSS^2$ and as a set of 2-flats in $\RR^4$.

The proof is based on double counting $I(\pts,\lines)$.
A line defined by $y=c(x-d)$ contains every point of $\pts$ of the form $(d+b,cb)$ for every $b\in A$.
That is, we have that $I(\pts,\lines) \ge |\lines||A| = n^3$.
For the rest of the proof we will derive upper bounds for $I(\pts,\lines)$.

We partition the incidences in $\pts\times\lines$ into two types, as follows.
We say that an incidence $(p,\ell)\in \pts\times\lines$ is \emph{special} if there exists a second line $\ell'\in \lines$ such that $\ell$ and $\ell'$ are members of the same line family, and $p$ is in the infinite intersection of this family.
If an incidence $(p,\ell)\in \pts\times\lines$ is not special, then we say that it is a \emph{standard} incidence.

We first bound the number of standard incidences.
By considering $\SSS^2$ as $\RR^4$, we obtain an incidence problem with two-dimensional flats.
If $(p,\ell_1)$ and $(p,\ell_2)$ are standard incidences, then $\ell_1\cap \ell_2 = \{p\}$.
These are regular incidences, as defined in Theorem \ref{th:IncR4}.
By that theorem, the number of standard incidences is $O\left(|\pts|^{2/3+\rho}|\lines|^{2/3}+|\pts|+|\lines|\right)$.

When the number of standard incidences is larger than the number of special incidences, we have that
\[ I(\pts,\lines) = O\left(|\pts|^{2/3+\rho}|\lines|^{2/3}+|\pts|+|\lines|\right) = O\left(|A+A|^{2/3+\rho}|AA|^{2/3+\rho}n^{4/3}+ |A+A||AA|\right). \]

Combining this with $I(\pts,\lines) \ge n^3$ leads to $|A+A||AA| = \Omega(n^{5/2-\rho})$.
This immediately implies the assertion of the theorem, for any $0\le \alpha < 1/2$.

\parag{Handling special incidences.}
It remains to consider the case where the number of special incidences is larger than the number of standard incidences.
We say that a special incidence $(p,\ell)$ corresponds to a line family $\lines^*$ if the line $\ell$ is contained in the line family and the point $p$ is in the infinite intersection of the family.
By the definition, each special incidence corresponds to at least one line family.
By Lemmas \ref{le:LineFamDouble} and \ref{le:OppositeSigns}, a special incidence can correspond to at most one positive family and to at most one negative family.
In the rest of the analysis, we assume that at least half of the special incidences correspond to a positive family.
The other case, in which at least half of the special incidences correspond to a negative family, is handled in a symmetric manner.

We remove the special incidences that are not associated with positive line family.
By the above assumption, this does not asymptotically change $I(\pts,\lines)$.
The removal process may turn some special incidences to standard incidences.
We bound the number of new standard incidences in the same way we bound the number of the original standard incidences.
As before, if most of the original special incidences became standard incidences, then we are done.

It remains to study the case where most of the original value of $I(\pts,\lines)$ comes from the remaining special incidences.
Denote by $I(\alpha',\beta,\gamma,\delta)$ the number of special incidences $(p,\ell)\in \pts\times \lines$ that satisfy the following.
Let $\lines^*$ be the positive line family that corresponds to $(p,\ell)$, let $(s,s')$ be the point parameter of $\lines^*$, and let $(t,t')$ be the line parameter of $\lines^*$.
\begin{itemize}
\item The number of lines of $\lines$ that satisfy $t=\Delta^-(a)$ and $t'=\Delta^-(b)$ is at least $n^{2\alpha'}$ and smaller than $2n^{2\alpha'}$.
\item The number of lines of $\lines$ that are in $\lines^*$ is at least $n^\beta$ and smaller than $2n^\beta$.
\item The number of points $(x,y)\in\pts$ that satisfy $s=\Delta^+(x)$ and $s'=\Delta^+(y)$ is at least $n^\gamma$ and smaller than $2n^\gamma$.
\item The pair $(s,s')$ is the point parameter of at least $n^\delta$ and fewer than $2n^\delta$ positive line families that satisfy the property stated in the second item.
\end{itemize}

Note that we can take $\Theta(\log^4 n)$ elements $I(\alpha',\beta,\gamma,\delta)$ such that every special incidence is counted in at least one of those elements.
Thus, the number of special incidences is at most the maximum size of $I(\alpha',\beta,\gamma,\delta)$ times $\Theta(\log^4 n)$.

We study some basic properties of the parameters $\alpha',\beta,\gamma,\delta$ that maximize $I(\alpha',\beta,\gamma,\delta)$.
For that purpose, we assume that $\alpha',\beta,\gamma,\delta$ are fixed.
We denote by $S$ the set of point parameters that participate in incidences of $I(\alpha',\beta,\gamma,\delta)$.
Let $F$ be the set of line families that contain at least $n^\beta$ lines of $\lines$ and fewer than $2n^{\beta}$ such lines.
Let $T$ be the set of line parameters of the families of $F$.
Since $|\lines|=n^2$ and every pair of $T$ corresponds to $\Omega(n^{2\alpha'})$ lines of $\lines$, we get that $|T| =
O(n^{2-2\alpha'})$.

We study how many lines of $\lines$ can correspond to a given line parameter $(t,t')\in T$.
Recalling that every line of $\lines$ is of the form $y=c(x-d)$, we note that $c$ uniquely determines $t$.
Since $A$ has multiplicity $n^\alpha$, there are $O(n^\alpha)$ choices for $c$.
Recalling from \eqref{eq:DeltaProp} that $\Delta^-(a\cdot b) = \Delta^-(a)\cdot \Delta^-(b)$, we get that for a fixed $c$ there are $O(n^\alpha)$ choices for $d$.
Thus, the number of lines in $\lines$ with a given line parameter is $O(n^{2\alpha})$.
This implies that $0\le \alpha'\le \alpha$.
We also have that $\beta \le 2\alpha'$, since otherwise there are not enough lines with the same line parameter to create a family in $F$.

We now consider how many lines from $\lines$ can be part of the same line family.
In general we consider lines defined by an equation of the form $y=(a_1+j a_2) x + (b_1+j b_2)$, and a line in $\lines$ is defined by an equation of the form $y=c(x-d)$ with $c,d\in A$.
By Lemma \ref{le:LineFamilyDoub}, all the lines in the same positive family have the same $\Delta^-(a)$ and $\Delta^-(b)$.
The multiplicity of $A$ implies that there are at most $n^\alpha$ possible values $c$.
Also by Lemma \ref{le:LineFamilyDoub}, either all the lines in a family have the same $b$, or they all satisfy relations of the form $b_1 = s(m-a_1)$ and $b_2 = s(m-a_2)$.
In either case, the values of $b$ are uniquely determined by the values of $a$.
That is, choosing $c$ uniquely determines $d$.
We conclude that every family contains at most $n^\alpha$ lines from $\lines$, so $0 \le \beta \le \alpha$.

Recall from \eqref{eq:DeltaProp} that $\Delta^+(a+b) = \Delta^+(a)+\Delta^+(b)$.
Since the multiplicity of $A$ is $n^\alpha$, at most $n^{1+\alpha}$ sums $x$ in $A+A$ can have the same value for $\Delta^+(x)$.
Similarly, since $\Delta^+(a\cdot b) = \Delta^+(a)\cdot \Delta^+(b)$, at most $n^{1+\alpha}$ products $y$ in $AA$ can have the same value for $\Delta^+(y)$.
Since $\pts= (A+A)\times (AA)$, at most $n^{2+2\alpha}$ points of $\pts$ can correspond to the same point parameter $(x,y)$ in  $\SSS^2$.
Thus, $0 \le \gamma \le 2+2\alpha$.
We also have the straightforward bound $n^\gamma \le |\pts|$, or equivalently $\gamma \le (\log |\pts|)/\log n$.

Next, we consider the maximum number of line families of $F$ that can have the same point parameter.
Recall that every line of a positive family with point parameter $(s,s')$ satisfies $s' = (a_1+a_2)s + (b_1+b_2)$ (for example, see the proof of Lemma \ref{le:LineFamilyDoub}(b)).
Since the lines of $\lines$ are defined as $y=c(x-d)$ with $c,d\in A$, for every choice of $c$ and a point parameter, the value of $b_1+b_2$ is uniquely determined.
This implies that a specific point parameter has $O(n^{1+\alpha})$ lines of $\lines$ corresponding to it.
We conclude that $0 \le \delta \le 1+\alpha - \beta$.

To recap:
\begin{align*}
0\le \alpha', \beta \le &\alpha, \qquad \beta \le 2\alpha', \qquad 0 \le \delta \le 1+\alpha-\beta, \\[2mm]
0 \le &\gamma \le \min\left\{2+2\alpha, (\log |\pts|)/\log n\right\}.
\end{align*}

We next bound the number of families in $F$.
Recall that $|T|=O(n^{2-2\alpha'})$, and that each pair of $T$ corresponds to fewer than $2n^{2\alpha'}$ lines of $\lines$.
For a fixed $(t,t')\in T$, by Lemma \ref{le:LineFamDouble} every two families corresponding to $(t,t')$ have at most one line in common.
There are fewer than $\binom{2n^{2\alpha'}}{2} = O(n^{4\alpha'})$ pairs of lines of $\lines$ that correspond to $(t,t')$.
Every pair of lines can appear together in at most one line family, and each line family subsumes at least $\binom{n^{\beta}}{2} = \Theta(n^{2\beta})$ such pairs.
Thus, the number of families that correspond to $(t,t')$ is $O(n^{4\alpha'-2\beta})$.
By summing up over every $(t,t')\in T$, we obtain that $|F| = O(n^{2+2\alpha'-2\beta})$.

Since each pair $(s,s')\in S$ corresponds to $\Theta(n^\gamma)$ points of $\pts$, we have $|S| = O(|\pts|/n^\gamma)$.
Since $|F| = O(n^{2+2\alpha'-2\beta})$ and each point parameter subsumes $\Theta(n^{\delta})$ families of $F$, we obtain $|S| = O(n^{2+2\alpha'-2\beta-\delta})$.

We think of a point parameter $(s,s')$ as corresponding to the 2-flat in $\RR^4$ defined by $s=x_1+x_2$ and $s'=y_1+y_2$.
By Lemma \ref{le:FamInHyper}, a line family is fully contained in a hyperplane in $\RR^4$, and two positive families are contained in the same hyperplane if and only if they have the same line parameter.
Let $H$ be a generic 2-flat in $\RR^4$, such that $H$ intersects every 2-flat that corresponds to a point parameter $(s,s')\in S$ in a single distinct point, and that $H$ intersects every hyperplane containing a line family at a distinct line.
Let $\pts_H$ be the resulting set of $|S|$ points in $H$ and let $\lines_H$ be the resulting family of $|T|$ lines in $H$.
By definition, every point of $\pts_H$ is incident to $\Omega(n^{\delta-2\alpha'+\beta})$ lines of $\lines_H$.
Recalling that $|T|=O(n^{2-2\alpha'})$, Theorem \ref{th:DualFormST} implies that
\[ |S|=O\left(\frac{(n^{2-2\alpha'})^2}{(n^{\delta-2\alpha'+\beta})^3} + \frac{n^{2-2\alpha'}}{n^{\delta-2\alpha'+\beta}}\right) = O\left(n^{4+2\alpha'-3\delta-3\beta} + n^{2-\delta-\beta}\right).\]

Recall that a point parameter $(s,s')\in S$ is associated with the plane in $\RR^4$ defined by $x_1+x_2=s$ and $x_3+x_4=s'$.
Denote this plane as $h$.
There are $\Theta(n^\delta)$ families with point parameter $(s,s')$, each intersecting in a common line in $h$.
There are $\Theta(n^\gamma)$ points of $\pts$ in $h$.
The intersection lines of the different families are distinct by definition.
By the Szemer\'edi--Trotter theorem, the number of incidences between these points and lines is $O(n^{2(\delta+\gamma)/3}+n^\delta+n^\gamma)$.
Since each line in the imaginary plane corresponds to $\Theta(n^\beta)$ lines of $\lines$, the number of incidences in the $h$ is
\begin{equation} \label{eq:IncInPntParam}
O\left(n^\beta\left(n^{2(\delta+\gamma)/3}+n^\delta+n^\gamma\right)\right).
\end{equation}

Note that \eqref{eq:IncInPntParam} is identical to \eqref{eq:IncInSpecialPnt}.
In addition, we obtained the exact same bounds for $\alpha',\beta,\gamma,\delta, |T|,|S|$, and $|F|$ as in the proof of Theorem \ref{th:DualMainCase}.
We may thus repeat the technical calculation at the end of the proof of Theorem \ref{th:DualMainCase}.
We do not repeat the entire calculation here.
As in the proof of Theorem \ref{th:DualMainCase}, this leads to $\max\{|A+A|,|AA|\} = \Omega^*\left(n^{3/2-5\alpha/8}\right)$.
\end{proof}

Proving the bound of Theorem \ref{th:DoubleSP} for the case where $1/2 \le \alpha < \kappa$ is identical to the proof of Corollary \ref{co:DualLargeAlpha}.
Thus, we do not repeat this proof here.

\subsection{Adapting Solymosi's argument to double numbers} \label{ssec:SolyDouble}

We now adapt Solymosi's sum-product argument \cite{Soly09} to sets of double numbers.

\begin{theorem}
Let $A$ be a set of $n$ dual numbers with multiplicity $n^{\alpha}$, for some $0\le \alpha<1/2$.
Then
\[ \max\{|A+A|,|AA|\}=\Omega^*\left(n^{(4-2\alpha)/3}\right). \]
\end{theorem}
\begin{proof}
The proof is similar to the proof of Theorem \ref{th:SolyDual}, with $\re(a)$ replaced by $\Delta^+(a)$.
Equations \eqref{eq:DeltaProp} illustrate that $\Delta^+(a)$ has the same arithmetic properties we used with $\re(a)$.
For $a\in \SSS$, we refer to $\Delta^+(a)$ as the \emph{parameter} of $a$.

For $\lambda \in \RR$, we define
\begin{align*}
r_{A}^{\times}(\lambda)&=\left|\left\{ (a,a')\in A^{2}\ :\ \Delta^+(aa')=\lambda\right\} \right|, \\[2mm]
r_{A}^{\div}(\lambda)&=\left|\left\{ (a,a')\in A^{2}\ :\ \Delta^+(a/a')=\lambda\right\} \right|.
\end{align*}
In other words, $r_{A}^{\times}(\lambda)$ is the number of ways to obtain $\lambda$ as the parameter of a product of two elements of $A$, and similarly for $r_{A}^{\div}(\lambda)$.

We repeat the pruning steps of $A$ as in the proof of Theorem \ref{th:SolyDual}, to obtain that every element of $A$ has a positive parameter.
We then repeat the multiplicative energy calculation from the proof of Theorem \ref{th:SolyDual}.
This implies that exists $0\le m<\log n$ with $\Lambda=\left\{ \lambda\in \Delta^+(A/A)\ :\ 2^m\le\r \lambda<2^{m+1}\right\}$ such that
\begin{equation*}
1>\frac{E^{\times}(A)}{|\Lambda|2^{2m+2}\log n}.
\end{equation*}
(The multiplicative energy $E^{\times}(A)$ is defined as in Section \ref{ssec:SolyDual}.)

Consider the planar point set ${\cal P}=A\times A\subset\SSS^{2}.$
Since ${\cal P}+{\cal P}=(A+A)\times(A+A)$, we have that $|{\cal P}+{\cal P}|=|A+A|^{2}$.

For each $1\le i\le|\Lambda|,$ let $\ell_{i}$ denote the line in $\RR^{2}$ defined by $y=\lambda_{i}x$.
Let $\pts\cap\ell_{i}$ be the set of points $(a,b)\in {\cal P}$ that satisfy $(\Delta^+(a),\Delta^+(b))\in \ell_{i}$ (equivalently, $\Delta^+(a)=\lambda_{i} \cdot \Delta^+(b)$).
By definition, for each of the $|\Lambda|$ lines we have $2^m \le |{\cal P}\cap \ell_{i}|<2^{m+1}$.
Let ${\cal P}\cap_{\RR}\ell_{i}$ be the set of points $(\Delta^+(a),\Delta^+(b))$ such that $(a,b)\in \pts\cap\ell_{i}$.
Note that ${\cal P}\cap\ell_{i}$ is in $\SSS^2$ while ${\cal P}\cap_{\RR}\ell_{i}$ is in $\RR^2$, and that $|{\cal P}\cap\ell_{i}|\ge |{\cal P}\cap_{\RR}\ell_{i}|$.

The lines $\ell_i \subset \RR^2$ are all incident to the origin.
In addition, for every $p\in \pts \cap\ell_{i}$ and $q\in \pts \cap \ell_{i+1}$, the point $\Delta^+(p+q)\in \RR^2$ lies in the interior of the wedge formed by $\ell_{i}$ and $\ell_{i+1}$ in the first quadrant of $\RR^{2}$.
Indeed, if positive $a,b,c,d\in \RR$ satisfy $a/b<c/d$, then $a/b<(a+c)/(b+d) <c/d$.
Thus, for any $1\le i< i'< |\Lambda|$, the sets $(\pts\cap_{\RR}\ell_{i})+(\pts \cap_{\RR}\ell_{i+1})$ and $(\pts \cap_{\RR}\ell_{i'})+(\pts \cap_{\RR}\ell_{i'+1})$ are disjoint.

Fix $1\le  i <|\Gamma|$.
For any $a_{1},a_{2}\in \pts \cap \ell_{i}$ and $a_{3},a_{4}\in \pts \cap \ell_{i+1}$, we have that $\Delta^+(a_{1}+a_{3})\ne \Delta^+(a_{2}+a_{4})$ unless $\Delta^+(a_{1})=\Delta^+(a_{2})$ and $\Delta^+(a_{3})=\Delta^+(a_{4})$.
Indeed, for variables $c,d\in \RR$, the system $(c,c\cdot\lambda_{i})+(d,d\cdot\lambda_{i+1})=(p_{x},p_{y})$ has a unique solution.
In other words, for any $p,q\in{\cal P}\cap \ell_{i}$ and $r,t\in{\cal P}\cap \ell_{i+1}$ that satisfy $\Delta^+(p)\neq \Delta^+(q)$ or $\Delta^+(r) \neq \Delta^+(t)$, we have $p+r\ne q+t$.
Since ${\cal P} = A\times A$ and since $A$ has multiplicity $n^\alpha$, for each $(r,t)\in \RR^2$ at most $n^{2\alpha}$ pairs $(a,b)\in \pts$ satisfy $(\Delta^+(a),\Delta^+(b)) = (r,t)$.
For each point in ${\cal P}\cap_{\RR}\ell_{i+1}$ we arbitrarily consider one point of ${\cal P}\cap\ell_{i+1}$ that corresponds to it, and denote the resulting set as $Q_i$.
Note that $|Q_i|\ge |{\cal P}\cap\ell_{i+1}|/n^{2\alpha}$ and that $Q_i$ consists of points with distinct parameters.
We claim that $|({\cal P}\cap \ell_{i})+ Q_i| = |{\cal P}\cap\ell_{i}| \cdot |Q_i|$.
In other words, we claim that every element of $({\cal P}\cap \ell_{i})+ Q_i$ can be written as a sum in a unique way.
Indeed, for $q\in Q_i$ and $a,a'\in {\cal P}\cap \ell_{i}$ we clearly have $a+s \neq a'+s$ when $\Delta^+(a)\neq \Delta^+(a')$.
If $\Delta^+(a)= \Delta^+(a')$ then $a$ and $a'$ have distinct imaginary parts, again implying $a+s \neq a'+s$.
This leads to
\[ \left|({\cal P}\cap \ell_{i})+({\cal P}\cap \ell_{i+1})\right| \ge \left|({\cal P}\cap \ell_{i})+Q_i\right| \ge |{\cal P}\cap \ell_{i}|\cdot |{\cal P}\cap \ell_{i+1}|/n^{2\alpha}. \]

The rest of the analysis is a technical calculation identical to the one at the end of proof of Theorem \ref{th:SolyDual}.
We do not repeat this analysis here.
\end{proof}


\end{document}